\documentclass[12pt]{article}
\usepackage[pdftex]{color}
\usepackage{latexsym}
\usepackage{amssymb}
\usepackage{amsmath}
\usepackage{amsthm}
\usepackage{rotating}
\usepackage{multirow}
\usepackage{mathrsfs}
\usepackage{wasysym}
\usepackage{empheq}
\usepackage{enumerate}
\usepackage[OT2,OT1]{fontenc}
\def\cyr{
\renewcommand\rmdefault{wncyr}
\renewcommand\sfdefault{wncyss}
\renewcommand\encodingdefault{OT2}
\normalfont
\selectfont}
\DeclareMathAlphabet{\zap}{OT1}{pzc}{m}{it}
\DeclareTextFontCommand{\textcyr}{\cyr}
\def\yha{\rho}
\def\be{\begin{equation}}
\def\ee{\end{equation}}
\def\bea{\begin{eqnarray*}}
\def\eea{\end{eqnarray*}}

\newcommand{\rad}{\varrho}

\def\CC{\mathbb C}

\def\Ric{r}

\def\Rm{{\mathcal R}}

\newtheorem{main}{Theorem}
\DeclareMathOperator{\tr}{tr}

\DeclareMathOperator{\Hess}{Hess}

\DeclareMathOperator{\vol}{Vol}
\DeclareMathOperator{\End}{End}
\DeclareMathOperator{\Sym}{Sym}

\newtheorem{thm}{Theorem}[section]
\newtheorem{lem}[thm]{Lemma}
\newtheorem{prop}[thm]{Proposition}
\newtheorem{cor}[thm]{Corollary}
\newtheorem{defn}{Definition}

\newenvironment{remark}{\medskip \noindent {\bf Remark.}}{\hfill $\diamondsuit$ \\}
\newenvironment{note}{\medskip {\bf Note.}}{\hfill \rule{.5em}{1em} \\}

\newenvironment{rmk}{\mbox{ }\\{\bf  Remark}\mbox{ }}{
\hfill \bigskip}
\def\ZZ{{\mathbb Z}}
\def\QQ{{\mathbb Q}}
\def\RR{{\mathbb R}}
\def\CP{{\mathbb C \mathbb P}}

\def\eucl{{\boldsymbol\delta}}

\begin{document}

\title{Desingularizations of  Conformally\\ K\"ahler, Einstein Orbifolds}

\author{\parbox{1.75in}{\center Claude LeBrun\thanks{Supported in part by   NSF grant DMS-2203572.}\\
{\scriptsize  Stony Brook University} }\quad and ~  \parbox{1.5in}{\center  Tristan Ozuch\thanks{Supported in part by   NSF grant
DMS-2405328.}\\
{\scriptsize   MIT}}}

\date{}
\maketitle

\begin{abstract} 
Let $\{ (M,g_i)\}_{i=1}^\infty$ be a  sequence of smooth compact oriented  Einstein $4$-manifolds of
fixed Einstein constant $\lambda > 0$ that 
Gromov-Hausdorff converges
to  a $4$-dimensional  Einstein orbifold $(X,g_\infty )$.  Suppose, moreover, that 
 the limit metric $g_\infty$ is {\em Hermitian} with respect to some complex structure on the limit orbifold $X$, 
 that   $X$ has at least one singular point, and that every gravitational   instanton
that bubbles off from   $(M,g_i)\to (X,g_\infty )$  is  {\em anti-self-dual}.  Then, for all sufficiently large $i$,  
the given  $(M,g_i)$ are all   K\"ahler-Einstein.  As a consequence, the limit orbifold $(X,g_\infty )$ is also 
K\"ahler-Einstein, and  must in fact be one of the  orbifold limits  classified \cite{sunspot} by Odaka, Spotti, and Sun.  
  \end{abstract}

\section{Introduction}

Three decades ago, Anderson \cite{andprior} and Bando-Kasue-Nakajima \cite{bando2,bando1,bakanaka,nakajima} \linebreak 
 discovered a generalization of Gromov's compactness theorem \cite{gromhaus}  that gives 
   $4$-dimensional Einstein {\sf orbifolds} a previously-unsuspected  central role 
   in the theory of smooth compact  Einstein $4$-{manifolds}. 
For example,   if we are given a  sequence $(M_i, g_i)$ of smooth compact connected Einstein $4$-manifolds of  fixed  Einstein constant $\lambda >0$ 
and fixed Euler characteristic, with   volumes $\vol (M,g_i)$  uniformly bounded
away from zero,     then there exists a  subsequence $(M_{i_j},g_{i_j})$ that converges in  the 
   Gromov-Hausdorff distance  
to a compact connected  $4$-dimensional Einstein orbifold $(X, g_\infty )$ with the same Einstein constant $\lambda$. 
After further refining our subsequence, we may moreover arrange \cite{andche} that the manifolds $M_{i_j}$ are all  diffeomorphic\footnote{The 
 phenomenon of diffeomorphism-finiteness that is operative here 
has more recently been shown  \cite[Theorem 1.12]{cn}   to  hold under dramatically    weaker hypotheses.} to
 some reference  smooth compact $4$-manifold $M$.
The singularities of the limit orbifold $X$  are moreover  isolated, and so 
 constitute a finite subset $\mathfrak{S}\subset X$, while their  complement $X-\mathfrak{S}$    is   diffeomorphic to a connected 
open set $U\subset M$. Furthermore, after moving the metrics in the sequence by a suitable sequence of self-diffeomorphisms of $M$,
 the weak-limit 
metric $g_\infty$ on $X-\mathfrak{S}$ actually  becomes  the $C^\infty$ limit on compact subsets of the given sequence of Einstein metrics restricted to  $U\subset M$. Meanwhile, 
the Riemannian diameter of each connected component of $M-U$ shrinks to zero, but the  $L^2$ norm of the curvature 
on each such component remains bounded away from zero.  
After suitably scaling  up these high-curvature regions,  
their pointed Gromov-Hausdorff limits then yield  Ricci-flat ALE spaces that are said to {\em bubble off} from $M$. One subtlety
is that curvature can accumulate at different length scales, so that some of these ALE limits  could themselves be orbifolds. However, 
any  orbifold singularity of such an ALE space in turn  arises from  a {\em deeper bubble}. This  hierarchy of scales  however  terminates in  finitely many steps, and the {\em deepest bubbles}
that thereby arise  are therefore  smooth.  By reversing  the process, one can then   reconstruct  the smooth manifold $M$ from $X$ by \cite{andprior,bando1}  cutting out a neighborhood 
of each orbifold singularity,  replacing it with  a diffeomorphic copy of the corresponding bubble, and then  iterating the process until all the orbifold singularities have been removed.

While the  above results  indicate that   certain Einstein orbifolds can arise as limits of smooth Einstein $4$-manifolds, the second author's previous results  \cite{ozu1,ozu2,ozu3} have definitively shown that not
 every Einstein orbifold arises in this way. The 
 present  article focuses on the   problem of   determining  which   compact $4$-dimensional {\sf K\"ahler-Einstein} orbifolds with $\lambda > 0$ 
can be expressed as limits of  sequences of  smooth  Einstein $4$-manifolds. If one  specializes the problem by requiring  that all  the smooth Einstein manifolds in the limiting sequence 
 be  {\sf K\"ahler-Einstein} manifolds, then 
a complete answer has  already   been given    by Odaka-Spotti-Sun \cite{sunspot}. 
But  if we now relax our assumptions, and   allow the metrics in the given sequence to have arbitrary holonomy, one might still 
hope    that the 
list of possible K\"ahler-Einstein   orbifold limits might remain unaffected.  Unfortunately, however, there is a  serious technical    obstacle to proving such a statement.   The root cause of  this difficulty is that   we currently lack a 
 complete classification of 
Ricci-flat ALE manifolds. 
 Indeed, all known examples of such ALE manifolds  are half-conformally-flat, so that, when correctly oriented,  they 
 satisfy $W^+\equiv 0$, where $W^+\in \End_0 (\Lambda^+)$ denotes the self-dual piece of the Weyl curvature tensor. While  a longstanding conjecture claims 
 that all Ricci-flat ALE $4$-manifolds are  in fact of this type, the  best available result in support of this conjecture  \cite{nakamass} assumes  
 the existence of a spin structure on  the universal cover of the ALE space for which one of the chiral spin bundles has trivial monodromy at infinity;
  but, unfortunately, 
 no one has yet discovered a compelling reason why  this assumption 
  should {\em a priori}
  always be valid. In this paper, our approach     will  skirt  this issue  entirely  via recourse to the following  definition:

\begin{defn} 
\label{ranger}  
Let $(M^4, g_i)$ be a 
sequence of compact oriented $4$-dimensional Einstein manifolds
that converges to a positive-scalar-curvature compact Einstein orbifold $(X^4, g_\infty)$ 
in the Gromov-Hausdorff sense.  We will then say that $(M, g_i)$  is an {\sf admissible sequence}, and that 
$(X, g_\infty)$ is an {\sf admissible orbifold limit},
if every  Ricci-flat ALE space $(Y, \mathsf{h})$ that bubbles off from the $(M, g_i)$ 
satisfies
 $W^+\equiv 0$ with respect to the induced orientation. \end{defn}
 
 While restricting  our attention to admissible sequences  will  of course  narrow the scope of our 
 investigation, it will nonetheless   allow  us to prove   the following rather satisfying results: 

\begin{main} \label{alpha} 
Let  $(X^4 , g_\infty)$ be  a compact K\"ahler-Einstein orbifold of Einstein constant $\lambda >0$, and let 
 $(M^4, g_i)$
be an admissible  sequence of compact oriented Einstein manifolds   with  $(X , g_\infty)$ as its orbifold limit.
Then, for all $i \gg 0$,    the $(M, g_i)$ are actually  all K\"ahler-Einstein. 
In particular,  $(X , g_\infty)$ must be one of the K\"ahler-Einstein orbifolds classified by    Odaka-Spotti-Sun. 
\end{main}

Indeed,  essentially  the same result will  also follow even if one merely assumes that  $g_\infty$ is {\em Hermitian}, in the  sense that 
$g_\infty (J\cdot, J\cdot )= g_\infty$ for some integrable orbifold complex structure $J$ on $X$. 

\begin{main} \label{beta} 
Let  $(X^4 , g_\infty)$ be  a compact connected 
Einstein orbifold of Einstein constant $\lambda >0$, and suppose   that $g_\infty$ is {\em Hermitian}
with respect to some integrable  orbifold complex structure $J$ on $X$. 
 If $(X , g_\infty)$ has at least one singular point, and if $(M, g_i)\to (X , g_\infty)$ is an admissible  sequence of Einstein manifolds
 converging to $X$, then   the $(M, g_i)$ are  K\"ahler-Einstein for all $i \gg 0$,  
 and 
  $(X , g_\infty)$ is one of the K\"ahler-Einstein  orbifold  limits classified by    Odaka-Spotti-Sun. 
\end{main}
%

\section{Singularities of Type T}

\label{type-t}

When $(X,g_\infty )$ is the orbifold limit  of some admissible sequence $(M, g_i)$ of Einstein $4$-manifolds, 
we will show,  in Proposition \ref{once}  below, that 
its orbifold singularities must all have the following special property: 


\begin{defn} 
\label{loner} 
An isolated orbifold singularity of an oriented  orbifold $X^4$ will be said to be  of  {\sf  type T}
if it has an open neighborhood  modeled on $\RR^4/\Gamma$ for some finite subgroup $\Gamma < SO(4)$
for which  $\RR^4/\Gamma$  is also the tangent-cone-at-infinity of an  
 oriented   {\sf anti-self-dual} Ricci-flat  ALE manifold $(Y^4,\mathsf{h})$.
\end{defn}

%

This definition can now
 be restated  in a  more concrete and effective form, as an immediate  consequence of  the results
  of \c{S}uvaina \cite{ioana2} and Wright \cite{wright}:

\begin{prop} 
\label{revealed} An oriented singularity $\RR^4/\Gamma$ is 
of {\sf type T} iff $\Gamma < SO(4)$ is conjugate to a subgroup of $U(2)$ of one of the two 
following types: 

\begin{itemize} 
\item Finite subgroups $\Gamma\subset SU(2)$. These are  classified by the simply-laced Dynkin diagrams,
and their images in $SO(3) = SU(2)/\ZZ_2$ are  classical symmetry groups of regular polygons or platonic solids.
\begin{center}
  \begin{tabular}{| l | c | r |}
    \hline
Dynkin Diagram  &      $\Gamma\subset SU(2)$& Classical Figure  \\ \hline \hline
  $A_k$ \hspace{.7cm} 
\begin{minipage}[c]{0.4in}
\setlength{\unitlength}{1ex}
\begin{picture}(50,5)(0,-2.5)
\put(-2,0){\circle*{1}}
\put(0, -0.1){$\ldots$}
\put(4, 0){\circle*{1}}
\put(6, 0){\circle*{1}}
\put(8, 0){\circle*{1}}
\put(-2,0){\line(1,0){1.5}}
\put(3,0){\line(1,0){1}}
\put(4,0){\line(1,0){2}}
\put(6,0){\line(1,0){2}}
\end{picture}
\end{minipage}
&    \raisebox{.01in}{Cyclic} $\cong \ZZ_{k+1}$&\raisebox{-.07in}{\includegraphics[height=.25in]{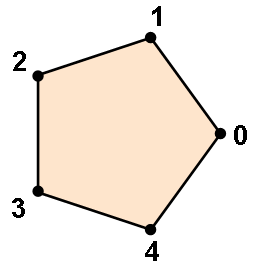}} \qquad  \qquad\\ \hline
 $D_k$ \hspace{.7cm}
\begin{minipage}[c]{0.4in}
\setlength{\unitlength}{1ex}
\begin{picture}(50,5)(0,-2)
\put(-2,0){\circle*{1}}
\put(0, -0.05){$\ldots$}
\put(4, 0){\circle*{1}}
\put(6, 0){\circle*{1}}
\put(8, 0){\circle*{1}}
\put(-2,0){\line(1,0){1.5}}
\put(3,0){\line(1,0){1}}
\put(4,0){\line(1,0){2}}
\put(6,0){\line(1,0){2}}
\put(6,0){\line(0,1){2}}
\put(6, 2){\circle*{1}}
\end{picture} 
\end{minipage}
&    \raisebox{.02in}{Binary Dihedral} &\raisebox{-.11in}{\includegraphics[height=.3in]{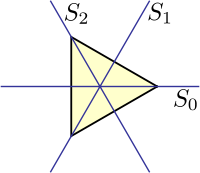}} \qquad  \qquad\\ \hline
  $E_6$  \hspace{.7cm}
\begin{minipage}[c]{0.4in}
\setlength{\unitlength}{1ex}
\begin{picture}(50,5)(0,-2)
\put(0,0){\circle*{1}}
\put(4, 2){\circle*{1}}
\put(2, 0){\circle*{1}}
\put(4, 0){\circle*{1}}
\put(6, 0){\circle*{1}}
\put(8, 0){\circle*{1}}
\put(0,0){\line(1,0){2}}
\put(2,0){\line(1,0){2}}
\put(4,0){\line(1,0){2}}
\put(6,0){\line(1,0){2}}
\put(4,0){\line(0,1){2}}
\end{picture} 
\end{minipage}
&  \raisebox{.02in}{Binary Tetrahedral} & \raisebox{-.11in}{\includegraphics[height=.3in]{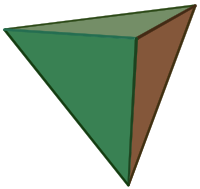}}  \qquad  \qquad\\ \hline
  $E_7$ \hspace{.7cm}
\begin{minipage}[c]{0.4in}
\setlength{\unitlength}{1ex}
\begin{picture}(50,5)(0,-2)
\put(-2,0){\circle*{1}}
\put(0,0){\circle*{1}}
\put(4, 2){\circle*{1}}
\put(2, 0){\circle*{1}}
\put(4, 0){\circle*{1}}
\put(6, 0){\circle*{1}}
\put(8, 0){\circle*{1}}
\put(-2,0){\line(1,0){2}}
\put(0,0){\line(1,0){2}}
\put(2,0){\line(1,0){2}}
\put(4,0){\line(1,0){2}}
\put(6,0){\line(1,0){2}}
\put(4,0){\line(0,1){2}}
\end{picture} 
\end{minipage}  &  \raisebox{.02in}{Binary Octahedral}& \raisebox{-.11in}{\includegraphics[height=.3in]{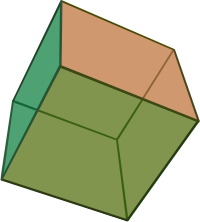}} \qquad  \qquad\\ \hline
 $E_8$ \hspace{.7cm}
\begin{minipage}[c]{0.4in}
\setlength{\unitlength}{1ex}
\begin{picture}(50,5)(0,-2)
\put(-4,0){\circle*{1}}
\put(-2,0){\circle*{1}}
\put(0,0){\circle*{1}}
\put(4, 2){\circle*{1}}
\put(2, 0){\circle*{1}}
\put(4, 0){\circle*{1}}
\put(6, 0){\circle*{1}}
\put(8, 0){\circle*{1}}
\put(-4,0){\line(1,0){2}}
\put(-2,0){\line(1,0){2}}
\put(0,0){\line(1,0){2}}
\put(2,0){\line(1,0){2}}
\put(4,0){\line(1,0){2}}
\put(6,0){\line(1,0){2}}
\put(4,0){\line(0,1){2}}
\end{picture} 
\end{minipage}
&  \raisebox{.02in}{Binary Icosahedral}& \raisebox{-.11in}{\includegraphics[height=.3in]{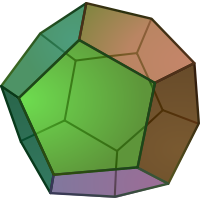}}  \qquad  \qquad\\ \hline  \end{tabular}
\end{center}

For each such $\Gamma$, 
 the corresponding quotient  $\CC^2/\Gamma$ is   {\em \cite{krontor}} the tangent cone at infinity
of a  family of  ALE  hyper-K\"ahler  $4$-manifolds, each of which is  obtained,  up to diffeomorphism,  by plumbing together copies
of $T^*S^2$ according to the intersection pattern dual to   the Dynkin diagram. 
\item Cyclic groups $\Gamma \cong \ZZ_{\ell m^2} < U(2)$, $m \geq 2$,  that are generated by  
$$
\left[\begin{array}{cc}\zeta &  \\& \zeta^{\ell mn-1}\end{array}\right] 
$$
for $\zeta=e^{2\pi i /\ell m^2}$, with   $n< m$ and  $\gcd (m,n)= 1$. Such  cyclic subgroups are often  indicated  {\em \cite{kolshep}} 
   by the shorthand  $\frac{1}{\ell m^2}(1,\ell m n-1)$.
For any such  subgroup,  
$\CC^2/\Gamma$ is  the tangent cone at infinity   of a
free $\ZZ_m$-quotient  of a hyper-K\"ahler manifold of type $A_{\ell m-1}$. Conversely,  every  non-simply-connected Ricci-flat ALE K\"ahler surface is  {\em \cite{ioana2}} one of these.
\end{itemize} 
\end{prop}

This of course  follows from the classification \cite{ioana2,wright} of oriented   Ricci-flat ALE $4$-manifolds  that are  {\sf anti-self-dual},  in the sense that  $W^+\equiv 0$. Here, one should especially take note of  
 Wright's striking theorem that 
any such space  is K\"ahler  \cite{wright}  with respect to some orientation-compatible complex structure.
This fact dovetails nicely with  our question, because  any singularity of {\sf type T} is therefore  modeled on $\CC^2/\Gamma$ for some finite group $\Gamma \subset U(2)$ 
that acts freely on the unit sphere $S^3\subset \CC^2$, and   any isolated singularity of an orbifold  K\"ahler surface 
is naturally   of   this form. 
 However,  notice that  our requirement  that $\CC^2/\Gamma$ must also be the tangent cone at infinity of an anti-self-dual 
  Ricci-flat ALE manifold  
 imposes further, less obvious  constraints.   
Indeed, in  the specific setting where  $X$ is an orbifold complex surface, singularities of  {\sf type T}  exactly coincide 
 \cite[Proposition 3.10]{kolshep} with the ones    algebraic geometers have traditionally called  singularities of   {\sf class T}.
From an algebro-geometric   perspective, the key property of these   singularities is that they  have
$\QQ$-Gorenstein local smoothings. However,  the   relevance  of this class of singularities  to our investigation 
instead primarily stems from  the following  differential-geometric observation:

\begin{prop} 
\label{once}
If $(X, g_\infty)$ is an admissible {orbifold limit}, then $X$ has only  singularities of {\sf  type T}.  
\end{prop}
\begin{proof} The  subtlety here  is that the ALE spaces that bubble off from a given sequence $(M, g_i) \to (X, g_\infty)$  might be orbifolds, rather than 
manifolds. However, when one surgically replaces each orbifold singularity of such an ALE space with the corresponding deeper bubble, and then iterates the 
process until it terminates, the resulting ALE manifold is then diffeomorphic to a Ricci-flat  ALE K\"ahler manifold \cite[Theorem 3]{bando1}. 
Thus, whenever an admissible  orbifold  limit  has a singularity modeled on $\CC^2/\Gamma$ for some 
$\Gamma < U(2)$, then  $\CC^2/\Gamma$ is also  the tangent cone at infinity of a {\sf non-singular}  Ricci-flat  ALE K\"ahler manifold.
\end{proof}

As an immediate consequence, we therefore have:

\begin{prop}
\label{twice}
Among compact K\"ahler-Einstein orbifolds $(X^4,g_\infty)$ that have only isolated singularities, 
there are infinitely many topological types  that 
 can never arise as limits of admissible sequences of Einstein $4$-manifolds. 
\end{prop}
\begin{proof} Let $g_{FS}$ be the Fubini-Study metric on $\CP_2$, and, for any prime number $p\geq 5$,  consider the finite group of isometries of $(\CP_2, g_{FS})$
that is  generated by 
$$
\left[\begin{array}{ccc} 1 &  &  \\ & \eta &  \\ &  & \eta^{-1}\end{array}\right],
$$
where $\eta = e^{2\pi i /p}$. The orbifold $\CP_2/\ZZ_p$ then has exactly three singular points, namely the images of $[1:0:0]$, $[0:1:0]$, and $[0:0:1]$, 
and each of these three singular points has a neighborhood modeled on  a quotient $\CC^2/\ZZ_p$, where the generator of the given $\ZZ_p$ action  on $\CC^2$ takes the form 
$$
\left[\begin{array}{cc}\eta &  \\ & \eta^{-1}\end{array}\right],
\left[\begin{array}{cc}\eta  &  \\ & \eta^{2}\end{array}\right]^{-1},  \mbox{ or } 
\left[\begin{array}{cc}\eta  &  \\ & \eta^2\end{array}\right],
$$
respectively.
The first of these cyclic groups is a subgroup of $SU(2)$, and so gives rise to  a singularity of {\sf type T}, modeled on the tangent cone at infinity of
 a gravitational instanton of type
$A_{p-1}$. However, the other two singularities are not of {\sf type T}, because in both cases the relevant  cyclic group  is not contained in 
$SU(2)$, and  its order is not divisible by a square. These K\"ahler-Einstein orbifolds $(\CP_2 , g_{FS})/\ZZ_p$ can therefore  never be   limits
of admissible   sequences  of smooth Einstein manifolds. 
\end{proof} 

Indeed, this set of examples also leads to a stronger  conclusion: 

\begin{prop} Either there are  infinitely many topological types of compact K\"ahler-Einstein orbifolds $(X,g_\infty)$ with isolated singularities
that cannot  arise as Gromov-Hausdorff limits of smooth Einstein manifolds, 
 or else there are Ricci-flat ALE $4$-manifolds that  remain to be discovered. 
\end{prop}
\begin{proof} If we assume that we already know all  possible  Ricci-flat ALE \linebreak $4$-manifolds, 
the only additional  freedom allowed beyond the  {\sf type T} scenario is that some of the Ricci-flat 
K\"ahler gravitational instantons might arise  with the non-standard orientation, thus 
allowing them to have $W^-=0$ instead of $W^+=0$. However, for any prime $p$, this would still only allow
$\ZZ_p\subset \mathbf{U}(2)$ to arise in connection to an orbifold singularity $\CC^2/\ZZ_p$ if the 
$\ZZ_p$ action were generated  by 
$$\left[\begin{array}{cc}\eta  &  \\ & \eta^{-1}\end{array}\right]  \mbox{ or } 
\left[\begin{array}{cc}\eta  &  \\ & \eta\end{array}\right]
$$
for some $p^{\rm th}$ root of unity $\eta$. 
However, for any prime  $p\geq 5$, two of the singularities of  the  isometric quotients $\CP_2/\ZZ_p$ discussed above are certainly not of this form. 
Thus, these examples cannot be Gromov-Hausdorff limits of sequences of smooth Einstein $4$-manifolds unless there are undiscovered ALE gravitational 
instantons that are not half-conformally-flat. 
 \end{proof}

Of course, the examples displayed  in the proof of Proposition \ref{twice}  are just the tip of the iceberg. Even among cyclic quotients of 
$\CP_2$, one could either replace the prime $p$ with an arbitrary product of distinct primes, or  change the specific powers of the root of unity 
$\eta$  in the matrix that  generates the action. Another   family of examples that would have sufficed  to prove Proposition 
\ref{twice} can be obtained by 
instead dividing the product Einstein metric on $\CP_1\times \CP_1$ by the action of a cyclic group $\ZZ_k$,  $k\neq 2, 4$, acting on both $S^2$ factors by identical rotations
about an axis; 
these examples have four orbifold singularities, but only two of them  are  of {\sf type T}. 
For an   infinite  class of non-locally-symmetric  K\"ahler-Einstein examples  with 
 singularities that are not of {\sf type T}, see \cite{bgk,collszek,kol}. 
For non-K\"ahler examples   of  compact  $4$-dimensional   Hermitian Einstein orbifolds which  have  singularities that are not of {\sf  type T},  see  \cite{apogau,bg,bryboc,gallaw}.

\section{Hermitian, Einstein Orbifolds} 
\label{heo}

While Theorem \ref{alpha} just concerns  the  case  where  $(X,g_\infty)$ is a K\"ahler-Einstein orbifold with $\lambda > 0$, 
 Theorem \ref{beta} weakens this hypothesis by merely requiring that the $\lambda > 0$ Einstein orbifold $(X,g_\infty)$  be  Hermitian. Let us therefore  point out  that this  weaker condition can be restated in various equivalent ways.

\begin{thm} 
\label{hemhaw}
A compact  oriented  
simply-connected 
$\lambda > 0$ Einstein  orbifold $(X^4,g)$ 
is {\em Hermitian}, in the sense that  $g=g (J\cdot , J\cdot )$ for some  integrable orientation-compatible  orbifold complex structure $J$ on $X$, 
if and only if  one of the following equivalent conditions holds:

\begin{enumerate}[\rm (i)]
\item \label{uno} The Einstein metric $g$ is {\em conformally K\"ahler},  meaning that   there is a smooth positive function 
$u: X\to \RR^+$ such that $\check{g}= u^2 g$ is 
 a K\"ahler metric with respect to  some  orientation-compatible integrable orbifold complex structure $J$ on $X$. 
\item \label{dos}  The self-dual Weyl curvature $W^+: \Lambda^+ \to \Lambda^+$ of $g$ has two equal, negative eigenvalues at almost every point
of $X$.
\item There is a  global self-dual harmonic $2$-form $\omega$ on the compact orbifold $(X,g)$ such that $W^+(\omega, \omega ) > 0$ at every
 point of $X$. 
\label{tres}
\item The self-dual Weyl curvature $W^+: \Lambda^+\to \Lambda^+$ satisfies $\det (W^+) > 0$ at every point of $X$. 
\label{quatro}  \end{enumerate}
\end{thm}
\begin{proof}
If the Einstein orbifold $(X,g)$  satisfies condition  \eqref{uno}, 
the Riemannian Goldberg-Sachs theorem  implies \cite{lebhem}  that it either  
satisfies condition \eqref{dos}, or else satisfies  $W^+\equiv 0$. However,  the  compactness of $X$ and the positivity of the 
Ricci curvature of $g$ together exclude  the latter possibility,
because 
the orbifold generalization of \cite[Theorem 1]{boy2}   and the fact that $b_1(X)=0$ 
would  then  imply  that $g$ is  globally conformal to a scalar-flat K\"ahler orbifold metric, 
contradicting the assumption that $g$ has positive scalar curvature.

The characterization \eqref{dos}
of conformally K\"ahler, Einstein metrics is due to Derdzi\'{n}ski \cite{derd}, and is strictly 
local; the proof moreover  shows that if this eigenvalue condition holds on a dense set, it 
actually holds everywhere. When this   happens, the K\"ahler metric $\check{g}$ is actually
{\em extremal} in the sense of Calabi; and, after possibly replacing it with a constant multiple, 
 one can arrange for the   scalar curvature $s_{\check{g}}$ of this extremal K\"ahler metric  to exactly equal  the positive
function $u$. 
 
If criterion \eqref{tres} is satisfied, then, since the orbifold $X$ is assumed to be compact, the   Weitzenb\"ock-formula  proof of  \cite[Theorem 1]{lebcake}  works exactly as in the manifold case. This then    shows that the conformally related metric 
for which  $|\omega|\equiv \sqrt{2}$ must actually be K\"ahler. 

Criterion \eqref{quatro}   was first proposed by
Peng Wu \cite{pengwu}. If $\alpha : X\to \RR$ is the positive eigenvalue of $W^+$, the proof in \cite{lebdet} shows that $h= \alpha^{2/3}g$
is K\"ahler, using a Wetizenb\"ock argument that works just as well on a compact  orbifold as it does on a compact manifold. 
\end{proof}

\section{Constraints on Diffeotype} 

Next, we  completely classify the smooth closed $4$-manifolds that   carry admissible  sequences of Einstein metrics for which the  limit   is 
conformally K\"ahler. 

\begin{thm}
\label{upon} 
Suppose that $M$ is  a smooth compact oriented $4$-manifold that admits   an admissible  sequence  $g_i$ of  
Einstein metrics   such that 
 $\{ (M, g_i)\}$  Gromov-Hausdorff  converges  to 
  a Hermitian,    Einstein   orbifold $(X^4,g_\infty)$ with Einstein constant $\lambda > 0$. 
Then $M$ is necessarily diffeomorphic to a Del Pezzo surface. In other words,  $M$ is diffeomorphic    either   to $S^2\times S^2$
or to 
a connected sum $\CP_2 \# \ell \overline{\CP}_2$ for some 
 $\ell\in \{ 0,1, \ldots , 8\}$.
In particular, 
$\pi_1(M)=0$, $b_+(M) =1$, and $b_-(M) \leq 8$. 
\end{thm}

\begin{proof}
As we previously saw in the proof of Proposition \ref{once}, $M$ must be diffeomorphic to the manifold obtained by replacing each orbifold singularity of $X$ with a Ricci-flat K\"ahler 
ALE manifold. Each such replacement is implemented  \cite{andprior,bando1,nakajima} by cutting out a small distance ball around a singular point $p_k\in X$, 
thereby introducing a boundary  modeled on a spherical space-form $S^3/\Gamma_k$, and then 
gluing in a  truncation of  a  Ricci-flat anti-self-dual ALE  manifold  $(Y_k,\mathsf{h}^{(k)})$ for which a neighborhood of infinity is modeled on $(S^3/\Gamma_k) \times \RR^+$. 
If $\Gamma_k \subset SU(2)$,  so that  $(Y_k,\mathsf{h}^{(k)})$ is hyper-K\"ahler, choose 
 a compatible integrable  almost-complex structure $J^{(k)}$ which is asymptotic to the obvious one on $\CC^2/\Gamma_k$; otherwise, 
 when $\Gamma_k \subset U(2)$ but $Y_k$ is not simply connected, there is still an integrable almost-complex structure $J^{(k)}$ that is 
 analogously asymptotic to the obvious choice. Because $\mathsf{h}^{(k)}$ is ALE of order $4$, we thus have \cite{bakanaka,lebkmass} 
$$\mathsf{h}^{(k)}_{ab} = \delta _{ab} + O(\varrho^{-4}), \qquad h^{(k)}_{ab,c} = O(\varrho^{-5}),$$
and
$$
\omega^{(k)}_{ab} = \omega^0_{ab} +  O(\varrho^{-4}), \qquad \omega^{(k)}_{ab,c} = O(\varrho^{-5}),
$$
where $\varrho$ is the usual Euclidean radius, $\omega^{(k)}$ is the K\"ahler form of $(Y_k, \mathsf{h}^{(k)}, J^{(k)})$  and  $\omega^0= dx^1\wedge dy^1 + dx^2 \wedge dy^2$ is the standard orbifold symplectic structure on $\CC^2/\Gamma_k$. 
A quantitative refinement of the standard Moser stability argument then shows \cite[Proposition 1.1]{lebkmass} that, for each $k$,    there is 
 a diffeomorphism $\Psi=\Psi^{(k)}$ defined on some asymptotic 
region $\varrho > {A}$ with 
$$\Psi^*\omega^{(k)} =\omega^0  \qquad \text{and} \qquad \Psi  =  \text{identity} + O(\varrho^{-3}).$$ 
 Thus, for a large constant $C\gg 0$ and a suitably small $\varepsilon > 0$, 
the composition of $\Psi$  with  multiplication by $C/\varepsilon$ defines 
 a symplectomorphism  between  the tiny annular region 
$\varepsilon < \varrho < 2 \varepsilon$ in $(\CC^2 /\Gamma_k, \omega_0)$ and the annular region in  
  $(Y_k, \varepsilon^2 C^{-2}  \omega^{(k)})$ that is the image under $\Psi$ of the annulus $C< \varrho < 2C$.  
 Identifying  the domain of this map with a tiny annulus about the orbifold singularity $p_k\in X$ then shows that 
  the K\"ahler form on $(X,g_\infty)$ can be extended across  $Y_k$ as a symplectic form. Since  the map
  used here to attach $Y_k$  to the regular part of $X$  is moreover isotopic to the one used to reconstruct $M$ from $X$ and the  $Y_k$, 
  applying this same procedure
 to all of the orbifold singularities of $X$ therefore  constructs a symplectic form $\omega$ on a manifold diffeomorphic to $M$. 
 
 It follows  that the smooth compact $4$-manifold $M$ admits both a symplectic form $\omega$ and an Einstein metric $g_i$ of positive Einstein constant. 
 This implies \cite{chenlebweb,lebcam} that $M$ must be diffeomorphic to a Del Pezzo surface. Indeed, since $M$ admits both a symplectic form and 
 a metric of positive scalar curvature, 
 Seiberg-Witten theory implies that it contains a pseudoholomorphic $2$-sphere of non-negative self-intersection, and so 
 must be symplectomorphic  to a rational or ruled complex surface \cite{liu1,ohno}. The fact that $M$  must    also \cite{hit}  satisfy the strict 
 Hitchin-Thorpe inequality $(2\chi + 3\tau )(M) = c_1^2 (M) >0$ then narrows the list of possibilities down to  the Del Pezzo diffeotypes. 
\end{proof}

This result is nominally sharp, since 
  every Del Pezzo $4$-manifold $M$  actually does 
  admit \cite{chenlebweb} conformally K\"ahler, positive-$\lambda$ Einstein metrics $g$, and therefore 
admits admissible {\sf constant} 
sequences $g_i = g$ with   limit  $(M,g)$.  However, this tautological construction of course   never yields sequences that 
converge to  singular orbifolds. In fact,  we will eventually see that  demanding that $X$ have at least one
 singular point actually winnows  the list down 
 to the four  diffeotypes  $\CP_2 \# \ell \overline{\CP}_2$,  $5\leq \ell \leq 8$. 
 
 We have already seen that most K\"ahler-Einstein orbifold complex surfaces certainly cannot arise as limits of admissible sequences
 of Einstein manifolds.  This is partly explained by the fact that  the first Chern class $c_1= c_1^{\rm orb}$ of such an orbifold $X$ 
is usually only  defined as an element of $H^2(X, \QQ )$, and that its self-intersection $c_1^2 (X)$ is therefore  in general    only a rational number.  
When the orbifold is the limit of an admissible sequence, however, the situation is drastically simpler.

 \begin{prop}
 \label{jobs}
  If  a Hermitian, $\lambda > 0$ Einstein orbifold $(X^4,J, g_\infty)$ is the Gromov-Hausdorff limit of an admissible sequence $(M,g_i)$
 of smooth Einstein $4$-manifolds $(M,g_i)$, then 
 $c_1^2(X,J) = (2\chi+ 3\tau )(M) \in \{ 1,2, 3, 4, \ldots , 9\}$. 
 \end{prop}
 \begin{proof} We have just seen that $M \approx (X-\{ p_k\} ) \cup \coprod_k Y_k$, where each $p_k$ is a singular point of $X$, and where 
  $Y_k$  is a corresponding gravitational instanton. Moreover, each overlap $(X-\{ p_k\} ) \cap Y_k$ is diffeomorphic to 
  $\RR \times (S^3/\Gamma_k)$ for an appropriate  finite group $\Gamma_k\subset U(2)$ that acts freely on $S^3$. Finally, we have seen that 
   $M$
 admits a symplectic structure $\omega$ that is symplectically isotopic to the K\"ahler form of $(J, g_\infty )$ on $X-\{ p_k \}$, and that 
 is symplectically isotopic to the  K\"ahler form of a Ricci-flat K\"ahler metric on each $Y_k$. This in particular implies that there exists
 a positive integer $m$ such that the $m^{\rm th}$ tensor power $K^{- m}$ of the anti-canonical line bundle $K^{-1}= \wedge^2 T^{1,0}$ of $(M,\omega)$ is 
 trivial on each $Y_k$, and  is bundle-isomorphic to the dual of the pluricanonical line bundle $K^{\otimes m}$ on the regular part $X-\{ p_k\} $
 of $X$. Let us now choose a generic smooth section of $f$ of $K^{-m}$  which is non-zero on each $Y_k$ and 
 meets the zero section of $K^{- m}$ transversely along a compatibly-oriented compact surface $\Sigma^2 \subset M^4$.
 The  homology class of $\Sigma$ in $M$ is then Poincar\'e dual to $m c_1 (M,\omega )\in H^2(M, \ZZ)$, while its
 homology class in $X$ is Poincar\'e dual to $m c_1^{\rm orb}(X, J)\in H^2 (X, \QQ )$. The self-intersection
 $[\Sigma ]\cdot [ \Sigma ]$ therefore computes both $m^2 c_1^2 (M, \omega )$ and $m^2 c_1^2 (X,J)$, and  we must consequently  have 
$c_1^2 (X, J) = c_1^2 (M, \omega ) = \langle p_1 (\Lambda^+) , M\rangle = (2\chi + 3\tau )(M)\in \{ 1, \ldots , 9\}$,
where the final   list  of possible values follows  immediately from  Theorem \ref{upon}.
  \end{proof}
  
 As a consequence, we now   immediately obtain  the following variation on a  result of Odaka-Spotti-Sun \cite{sunspot}: 

\begin{prop} 
\label{birdy} 
Suppose that the  $\lambda > 0$ K\"ahler-Einstein orbifold $(X^4,J, g_\infty)$ 
is the Gromov-Hausdorff limit of an admissible sequence $(M,g_i)$
 of smooth Einstein $4$-manifolds $(M,g_i)$. If a singular point of $(X,J)$ 
 has a neighborhood modeled on $\CC^2/\Gamma_k$, where 
 $\Gamma_k\subset U(2)$, then 
 $$|\Gamma_k| < \frac{12}{(2\chi + 3\tau) (M)}.$$
In particular, 
 every singular point has $\Gamma_k =  \ZZ_2$  if $(2\chi + 3 \tau )(M) \geq 4$. 
 \end{prop}
\begin{proof} If $g_0$ is the metric induced on $S^4/\Gamma_k$ by the usual round metric on $S^4$, then 
the Bishop-Gromov inequality and Gauss-Bonnet  imply that  
$$32\pi^2 c_1^2(X) =   \int_X s^2_{g_\infty} d\mu_{g_\infty} <  \int_{S^4/\Gamma_k} s^2_{g_0} d\mu_{g_0}= \frac{24\cdot 8\pi^2\chi (S^4) }{|\Gamma_k |}= \frac{32\pi^2\cdot 12 }{|\Gamma_k |},$$
 where 
the  inequality is necessarily strict because the $\lambda> 0$ K\"ahler-Einstein metric $g_\infty$ cannot have constant sectional curvature.
The result thus follows, because  $c_1^2 (X) = (2\chi + 3\tau ) (M)$ by Proposition \ref{jobs}. 
\end{proof} 

Of course, Theorem \ref{beta} makes the stronger claim that this same conclusion actually  holds if we merely assume that the limit
Einstein orbifold $(X, g_\infty )$ is {\em Hermitian} --- or equivalently, by Theorem \ref{hemhaw}, that it is {\em conformally K\"ahler}.
 Moreover, once we have proved Theorem \ref{beta},  it will then follow that a Hermitian, Einstein orbifold limit of  an admissible sequence $(M,g_i) \to (X, g_\infty )$ can only 
 form singularities if $c_1^2 (M)=c_1^2 (X)\leq 4$; and, as we already know by  Theorem \ref{upon}, the latter is equivalent to requiring that $M\approx \CP_2 \# k \overline{\CP}_2$,   where 
 $5 \leq  k \leq 8$. 
  Indeed, the Odaka-Spotti-Sun classification \cite{sunspot} will then go on to predict that the limit of an admissible sequence can only 
  form  singularities of the following types:

\begin{itemize}
\item 
${   \CP_2}{   \# 5}  {   \overline{\CP}_2}$: only ${   A_1}$ singularities.
\item 
${   \CP_2}{   \# 6}  {   \overline{\CP}_2}$: only ${   A_1}$ or  ${   A_2}$  singularities.
\item 
${   \CP_2}{   \# 7}  {   \overline{\CP}_2}$: only ${   A_1}$, \ldots ,  ${   A_4}$  or ${   \frac{1}{4}(1,1)}$  singularities.
\item 
${   \CP_2}{   \# 8}  {   \overline{\CP}_2}$:  ${   A_1}$, \ldots ,  ${   A_{10}}$, ${   D_4}$,
 ${   \frac{1}{4}(1,1)}$, 
   ${   \frac{1}{8}(1,3)}$, or 
    ${   \frac{1}{9}(1,2)}$ singularities. 
\end{itemize} 
This striking consequence  helps illustrate the interest of  Theorems \ref{alpha} and  \ref{beta}. 

\section{Admissible Limits
and Desingularizations}

Our key to  understanding  Gromov-Hausdorff convergence of admissible sequences of Einstein metrics is based on 
\cite[Corollary 4.16]{ozu2}, which  we next present  in a  streamlined   form that has been adapted to  present  purposes:

\begin{thm}[Ozuch] \label{concrete} 
Given any positive real constant $\nu > 0$, there is a constant $\delta = \delta (\nu) > 0$ such that, 
  whenever $(M, g^{\mathcal{E}})$ is a smooth compact   Einstein $4$-manifold with 
\begin{itemize}
\item total volume  $\geq  \nu$,  and 
\item Ricci curvature $= 3$, 
\end{itemize} 
\noindent  and for which  there is an Einstein $4$-orbifold $(X,g_\infty)$ with 
$$d_{GH} ( (M, g^{\mathcal{E}} ), (X, g_\infty )) <  \delta,$$
then $(M, g^{\mathcal{E}})$ is obtained from $(X, g_\infty)$ by  the gluing-perturbation procedure  detailed in \S \ref{basics} below. 
\end{thm}

Here we have simplified the  discussion  by  homothetically rescaling our $\lambda > 0$ Einstein metrics in order to arrange for them 
to  have $\lambda = 3$, which then guarantees,  by Myers' theorem,  that they all have diameter $\leq \pi$. This simple observation 
then guarantees that  Theorem \ref{concrete} is  a straightforward  consequence  of \cite[Corollary 4.16]{ozu2}.

\bigskip

The gluing-perturbation procedure cited  in  Theorem \ref{concrete} 
consists of first creating a smooth manifold by iteratively replacing each orbifold singularity with the corresponding the gravitational instanton that bubbles off,
and then equipping the resulting  manifold with a family of so-called Einstein-modulo-obstruction metrics that  satisfy a weakened form of the Einstein condition.
Following \cite{ozu1,ozu2}, the manifolds that arise by this procedure will often be called {\em na\"{\i}ve desingularizations} of the original orbifold. 
For a quantitative discussion of these technical points, see \S \ref{basics} below. 

\bigskip 

In the absence of further information, however, the construction of    na\"{\i}ve desingularizations depends on a detailed inventory 
of the specific 
trees of gravitational instantons that are to replace  each orbifold singularity, along with scaling data that determines the relevant  rates at which the
relevant gravitational instantons are to bubble off. However, 
 in the context of 
admissible sequences, the  allowed bubbling modes at each singularity actually form a single continuous family, and any 
 resulting desingularization therefore belongs to a single predetermined diffeotype, which the present paper's second author has 
elsewhere    \cite{ozu1,ozu2} called the 
{\em expected topology}.  Combining  the framework of \cite[\S 5]{ozu2} with Theorem \ref{upon} therefore now provides us with the basic scaffolding 
needed to support our proofs of Theorems \ref{alpha} and \ref{beta}:

\bigskip 

\begin{prop} 
\label{universal}
Let $(X^4, g_\infty)$ be a compact Hermitian, Einstein orbifold with $\lambda > 0$  and only isolated singularities of
{\sf type T}. Then there is  a single family of na\"{\i}ve desingularizations of $(X, g_\infty)$ by Einstein-modulo-obstruction manifolds diffeomorphic to a Del Pezzo surface such that any admissible sequence
$(M, g_i)$ of Einstein $4$-manifolds that Gromov-Hausdorff converges to $(X, g_\infty)$ consists, for sufficiently large $i$, 
of na\"{\i}ve desingularizations
that belong to this family. 
\end{prop}  

\section{The Role of Wu's Criterion}

Our proofs of   Theorems \ref{alpha} and \ref{beta} will  also hinge  on 
 showing that, for all $i \gg 0$, the Einstein manifolds  $(M, g_i)$ must 
everywhere  satisfy 
 Wu's criterion 
\begin{equation}\label{pengwu}
    \det(W^+) >0, 
\end{equation}
where the self-dual Weyl curvature $W^+$ is considered as an endomorphism of the $3$-dimensional vector space $\Lambda^+\subset \Lambda^2$ of self-dual 
$2$-forms at each point of $M$. As was first emphasized by Peng Wu  \cite{pengwu}, this conformally-invariant condition is satisfied
by any conformally K\"ahler, Einstein metric on a Del  Pezzo surface. Conversely, the first author \cite{lebdet} and Wu \cite{pengwu} proved,
by two entirely different methods, 
 that  any Einstein metric on a compact simply-connected manifold that  satisfies
\eqref{pengwu} everywhere must actually be  conformally K\"ahler. Since \eqref{pengwu}  is an   \textit{open} condition, this in particular means that 
smooth 
 Einstein perturbations of  Hermitian, Einstein metrics always remain Hermitian. The gist of Theorems \ref{alpha} and \ref{beta}  is 
 essentially that  the same  also holds 
  for admissible sequences  that merely tend, in the \textit{Gromov-Hausdorff} sense,  to  an Einstein orbifold 
that satisfies any of the equivalent conditions of Theorem \ref{hemhaw}. 
To achieve this goal, we will use Proposition \ref{universal} to show that   condition \eqref{pengwu} continues to hold for such perturbations by   carefully controlling the  possible degenerations of Einstein 
$4$-manifolds \cite{ozu3,ozu1,ozu2} in a manner that provides strong information regarding  their curvature  \cite{biq2}.
Our proof will  thus  unfold  in accordance with  the following rough outline:
\begin{enumerate}[(a)]
    \item Construct suitable approximations $g^A_{t_i,v_i}$ of the Einstein metrics $g_i$.
    \item Show that these approximations $g^A_{t_i,v_i}$ satisfy condition \eqref{pengwu} for $i \gg 0$.
    \item Show that condition  \eqref{pengwu}  then also holds for  the given $g_i$ when $i \gg 0$.
    \item Invoke  \cite{lebdet,pengwu} to  conclude  that these $g_i$ must therefore  be  Hermitian. \end{enumerate}

The rest of our article is therefore organized as follows. In Section \ref{basics}, we will review  some key  technical facts  that underpin  our 
construction. In Section 
\ref{simplest}  we then 
prove Theorems \ref{alpha} and \ref{beta} in the prototypical case where there is only one singularity, and  trees of singularities do not arise. In Section 
\ref{indicative}, we then generalize this to allow for multiple singularities,  but still in the absence of 
bubble-trees. Finally, in Section \ref{forest}, we  prove our results in  the most general but  technically-challenging  setting, where the formation of bubble-trees 
is   allowed.

\section{Desingularizing Einstein Orbifolds}
\label{basics} 

%

Recall that the {\em Einstein tensor} of a Riemannian $n$-manifold $(M,g)$ is the  symmetric tensor field defined by 
$$\mathscr{E} = r - \frac{s}{2}g,$$
where $r$ and $s$ are, respectively,  the Ricci tensor and scalar curvature of $g$; this definition  is chosen so that    the differential  Bianchi identity
implies that $E$ is 
 automatically {\em divergence-free} with respect to $g$. 
In the  $n=4$ case that  is our sole concern here, the Einstein tensor  takes the simple and attractive  form  
 $$\mathscr{E} = \mathring{r}  - \frac{s}{4} g$$
when expressed in terms of the trace-free part  $\mathring{r}$  of the Ricci tensor $r$,
thereby allowing the Einstein equation  
$$r= \lambda g$$
for a chosen   Einstein constant $\lambda$ to  be rewritten as 
$$\mathscr{E} + \lambda g=0.$$
However,  because  of  the diffeomorphism-invariance of the Einstein condition, this last equation fails to be  elliptic. To remedy 
this, we now choose some background metric $g_0$, and then define a ``gauged and   normalized'' version 
\begin{equation}
\label{gaugedEinstein}
\mathscr{F}_{g_0}(g) = \mathscr{E} (g) + \lambda g + \delta_{g}^*\delta_{g_0} (g)
\end{equation}
of the Einstein tensor. On a compact manifold $M$, the elliptic equation
\begin{equation}
\label{pressure}
\mathscr{F}_{g_0}(g)=0
\end{equation}
is then equivalent to the pair of conditions 
\begin{empheq}[box=\fbox]{align}
 \phantom{\int}  \mathscr{E}(g)+\lambda g &= 0,\\
   \delta^*_g \delta_{g_0}g &= 0.  \phantom{\int}
    \label{pressure-gauge}
\end{empheq}
Indeed, since  $\delta_g\mathscr{E}=0$,  applying $\delta_g$ to  \eqref{pressure}  yields $\delta_g\delta^*_g \delta_{g_0} g=0$,  and this then implies that 
$$\int_M \| \delta^*_g \delta_{g_0} g \|^2 \, d\mu_g = \int_M \langle \delta_{g_0} g ,\,  \delta_g \delta^*_g \delta_{g_0} g \rangle \, d\mu_g =0.$$
Since $\ker \delta_g^*$ exactly consists of the Killing fields on  $(M, g)$, equation  
  \eqref{pressure-gauge}  is equivalent to saying  that $\delta_{g_0}g$ is a Killing field of $g$. 
This  formalism becomes particularly elegant  and natural if $g$ is known, for some {\em a priori} reason, to have discrete isometry group,
and the reader is therefore  invited to provisionally focus on the case in which neither the limit orbifold nor  the  smooth metrics  in the given sequence  admit any Killing vector fields. However, the presence of Killing fields merely leads to a mild failure in uniqueness when representing 
a given Einstein geometry by a solution of \eqref{pressure}. While this will have no impact on the regularity-type results that are our goal here, 
the interested reader may wish to consult   \cite{ozu2}  for further discussion of this issue. 

\medskip 

In dimension four,  
a  straightforward (albeit  tedious) calculation   shows  that 
 the linearization 
$\bar{P}_{g_0} = d\mathscr{F}|_{g_0}$ 
of \eqref{gaugedEinstein} at $g_0$ is the elliptic operator 
\begin{eqnarray}
\label{eq:def bar P}
        \bar{P}_{{g_0}}{h} &=& \textstyle{\frac{1}{2}}\Big[\nabla^*\nabla \mathbf{h}  - \Hess (\tr h)
        -2\mathring{\mathcal{R}} (h) +2\Sym (r\circ h)
        \nonumber\\ 
        && 
        -(  \Delta  \, \textup{tr} \, h  -\nabla^* \delta h -\langle \Ric , h \rangle ){g_0} +(2\lambda -s) h\Big], 
\end{eqnarray}
where the relevant operators, curvatures, traces, and inner products 
 are all defined with respect to $g_0$; here $\Hess$ and $\Delta =d^*d$ respectively denote the Hessian and positive-spectrum Laplacian on scalar-valued functions, 
 while the  Riemann curvature tensor acts on symmetric $2$-tensors  by 
$$[\mathring{\mathcal{R}}(h)]_{ab}:= {{{\mathcal{R}^c}_a}^d}_b~h_{cd}.$$ 
 In particular, when  $g_0$ is Einstein, with Einstein constant $\lambda$, 
   and when the symmetric $2$-tensor is $h$ 
    is traceless and divergence-free,   the 
linearization $\bar{P}_{g_0}$ then  reduces to 
\begin{equation}\label{eq: def P}
    P_{g_0}:= \frac{1}{2}\nabla^*\nabla - \mathring{\mathcal{R}}
\end{equation}
or, in other words,  half the Lichnerowicz Laplacian \cite[p. 54]{bes} of $g_0$. 

\begin{rmk} Of course, these operators can naturally be defined on symmetric $2$-tensors belonging to various function spaces.
On orbifolds, ALE spaces and their gluings, the theory developed in \cite{ozu3,ozu1,ozu2} actually considers them in the function spaces $C^{k,\alpha}_\beta$, $C^{k,\alpha}_{\beta,*}$ and $C^{k,\alpha}_{\beta,**}$ described in  Appendix  \ref{function spaces}.
\end{rmk}
 

In order to begin  implementing this last remark, we  now focus on the case of   an oriented $4$-dimensional 
orbifold $(X,g_\infty)$  with only finitely many  singularities, all  of {\sf type T}. Recall that, according to Definition \ref{loner}, 
this means that each singular point  $p_k\in X$ has a neighborhood modeled on a neighborhood of the origin in a 
 quotient   $\RR^4/\Gamma_k$, $\Gamma_k < SO(4)$, 
that also occurs  as the tangent-cone-at-infinity of some Ricci-flat ALE manifold with $W^+\equiv 0$. We also remind the reader that a
complete list  of the $\Gamma_k$ allowed  by this definition is provided  by  Proposition \ref{revealed}. Now notice that Definition \ref{loner} 
tells us that it 
makes sense to modify $X$ by cutting out a neighborhood of $p_k\in X$ and replacing it with an anti-self-dual Ricci-flat ALE manifold. 
In order to carry out this procedure  in a systematic way, we will next  establish some conventions for doing so.

\begin{defn}[The function $\rad_o$ on an orbifold]\label{ro} Let $(X,g_\infty )$ be a complete oriented $4$-dimensional orbifold with only  finitely many isolated singularities.
Choose an  $\epsilon_0\in  (0,1)$ such that the Riemannian $\epsilon_0$-distance ball $U_i=B_{\epsilon_0}(p_k)$ around each singular point $p_k$ is orientedly diffeomorphic, via the exponential map,  to $B^4/\Gamma_k$ for
some finite group $\Gamma_k< {SO}(4)$ that acts freely on $S^3$. 
    We then choose    a smooth function $\rad_o: X\to [0,1]$ which, on each $U_i$,   coincides with the Riemannian distance 
    from the singular point $p_k$, and which satisfies $\rad_o \geq \epsilon_0$ on the complement of the union of the $U_i$. For every 
    $\epsilon \in (0,\epsilon_0)$, we then also set 
    $$X(\epsilon):= \rad_o^{-1} [\epsilon, 1] = X -  \Big(\bigcup_k B_\epsilon (p_k)\Big)$$
    where $ B_\epsilon (p_k)\subset U_i$ is the radius-$\epsilon$ Riemannian distance ball about $p_k$. 
\end{defn}

Next, we  define $\mathfrak{O}(g_{\infty})$ to be the space of traceless, divergence-free symmetric $2$-tensors on a $4$-dimensional manifold or orbifold with metric $g_{\infty}$ 
 that  belong to the kernel of $P_{g_{\infty}}$. Then, as explained in \cite[Section 5]{ozu2},  there is a one-to-one map from a neighborhood of $0\in\mathfrak{O}(g_{\infty})$ to the moduli space of \textit{Einstein-modulo-obstructions} perturbations of $g_{\infty}$ defined by \eqref{pressure}:

\begin{prop}
    For each $v_o\in \mathfrak{O}(g_\infty )$, there is a unique  metric $\Bar{g}_{v_o}$ satisfying: 
    \begin{enumerate}
        \item $\mathscr{F}_{g_{\infty}}(\Bar{g}_{v_o}) \in \mathfrak{O}(g_{\infty}),$
        \item $\|\Bar{g}_{v_o}-g_\infty\|_{C^{2,\alpha}}\leqslant C \|v_o\|_{C^{2,\alpha}}^2$, and
        \item $\Bar{g}_{v_o}-(g_\infty+v_o)\perp_{L^2(g_\infty)}\mathfrak{O}(g_\infty)$.
    \end{enumerate}
    for some $C>0$.
  In what follows,  this  metric $\Bar{g}_{v_o}$  will be called an  {\em Einstein-modulo-obstructions} metric.
\end{prop}


We next recall some important key facts regarding  ALE Ricci-flat metrics.

\begin{defn}[The function $\rad_b$ on an ALE orbifold]\label{rb}
We define $\rad_b$ a smooth function on $N$ satisfying $\rad_b:= (\Psi_k)_* \yha_{\eucl}$ on each $U_k$, and $\rad_b:= (\Psi_\infty)_* \yha_{\eucl}$ on $U_\infty = N -  K$, and such that $\epsilon_0\leqslant  \rad_b\leqslant \epsilon_0^{-1}$ on the rest of $N$.
For $0<\epsilon\leqslant\epsilon_0$, we then set
 $$N(\epsilon):= \{\epsilon<\rad_b<\epsilon^{-1}\} = N -   \Big(\bigcup_k \Psi_k\Big(\overline{B_\eucl(0,\epsilon)}\Big) \cup \Psi_\infty \Big((\mathbb{R}^4\slash\Gamma_\infty) -  B_\eucl(0,\epsilon^{-1})\Big)\Big).$$
\end{defn}

Finally, we define $\mathfrak{O}(g_b)$ to be the space of infinitesimal Einstein deformations of $g_b$ which are traceless and divergence-free. 
As is explained in \cite[Section 5]{ozu2},  there is a one-to-one map from a neighborhood of $0\in\mathfrak{O}(g_b)$ to the moduli space of K"ahler \textit{Ricci-flat ALE metrics} perturbations of $g_b$.

\begin{defn}
    For each $v_b\in \mathfrak{O}(g_b)$, and for a suitable $C>0$, we let $\Bar{g}_{v_b}$ denote 
   the unique Einstein-modulo-obstructions metric satisfying: 
    \begin{enumerate}
        \item $\mathscr{F}_{g_b}(\Bar{g}_{v_b}) =0$,  
        \item $\|\Bar{g}_{v_b}-g_b\|_{C^{2,\alpha}_\beta}\leqslant C \|v_b\|_{C^{2,\alpha}_\beta}^2$, and
        \item $[\Bar{g}_{v_b}-(g_b+v_b)]\perp_{L^2(g_b)}\mathfrak{O}(g_b)$.
    \end{enumerate}
\end{defn}

\label{reecriture-controle}

We next  review the notion of a na\"{\i}ve desingularization of an orbifold, which played a crucial  role in the results of  \cite{ozu1,ozu2}.

\paragraph{Gluing of ALE spaces to orbifold singularities.}


    Let $0<2\epsilon<\epsilon_0$ be a fixed constant, $t>0$, $(Y,g_\infty)$ an orbifold and $\Phi: B_\eucl(0,\epsilon_0)\subset\mathbb{R}^4\slash\Gamma \to U$ the inverse of a local chart  around a singular point $p\in Y$. Let also $(N,g_b)$ be an ALE orbifold asymptotic to $\mathbb{R}^4\slash\Gamma$, and $\Psi: (\mathbb{R}^4\slash\Gamma)- B_\eucl(0,\epsilon_0^{-1}) \to N- K$ the inverse of an orbifold chart at infinity.
    
   For  $\mathsf{s}>0$, we let  $\phi_\mathsf{s}: \mathbb{R}^4\slash\Gamma \to \mathbb{R}^4\slash\Gamma$ denote the 
   rescaling map   $x\to \mathsf{s}x$. For $t<\epsilon_0^4$, we then define $Y\#N$ as $N$ glued to $Y$ via  the diffeomorphism $$ \Phi\circ\phi_{\sqrt{t}}\circ\Psi^{-1} : \Psi(A_\eucl(\epsilon_0^{-1},\epsilon_0t^{-1/2}))\to \phi(A_\eucl(\epsilon_0^{-1}\sqrt{t},\epsilon_0)).$$ We moreover fix some  $C^\infty$ cut-off function $\Upsilon:[0, \infty )\to [0,1]$ 
  supported in $[0,2]$ and equal to $1$ on $[0,1]$.
    
\begin{defn}[Na\"{\i}ve gluing of an ALE space to an orbifold]\label{def naive desing}
    We define a \emph{na\"{\i}ve gluing} of $(N,g_b)$ at scale $0<t<\epsilon_0^4$ to $(Y,g_\infty)$ at the singular point $p$, which we will denote $(Y\#N,g_\infty\#_{p,t}g_b)$ by putting $g_\infty\#_{p,t}g_b=g_\infty$ on $Y -  U$, $g_\infty\#_{p,t}g_b=tg_b$ on $K$, and 
    \begin{equation}\label{naive gluing metrics}
        g_\infty\#_{p,t}g_b :=  \Upsilon(t^{-\frac{1}{4}}\yha_{\eucl})\Psi^*g_b + \Big(1-\Upsilon(t^{-\frac{1}{4}}\yha_{\eucl})\Big)\Phi^*g_\infty
    \end{equation}
    on $\mathcal{A}(t,\epsilon):=A_\eucl(\epsilon^{-1}\sqrt{t},2\epsilon)$. Given two tensors $h_o$ and $h_b$ on $Y$ and $N$, we similarly define the \textit{na\"{\i}ve gluing of the tensors}:
    \begin{equation}\label{naive gluing tensors}
        h_o\#_{p,t}h_b :=  \Upsilon(t^{-\frac{1}{4}}\yha_{\eucl})\Psi_\infty^*h_b + \Big(1-\Upsilon(t^{-\frac{1}{4}}\yha_{\eucl})\Big)\Phi^*h_o
    \end{equation}
\end{defn}

%

 Let $(Y,g_\infty)$ be an Einstein orbifold, and $(M,g^D)$ a na\"{\i}ve desingularization of $(Y,g_\infty)$ by a tree of ALE Ricci-flat orbifolds $(N_j,g_{b_j})$ glued at scales $T_j>0$. 

Here, the manifold $M$ is also covered as $M = Y^t\cup \bigcup_jN_j^t$, where
\begin{equation}\label{def Mot}
    Y^t: = Y  -  \Big(\bigcup_k \Phi_k(B_\eucl(0,t_k^{\frac{1}{4}})) \Big),
\end{equation}
 where $t_k>0$ is the relative gluing scale of $N_k$ at the singular point $p_k\in Y$, and where 
 \begin{equation}\label{def Njt}
     N_j^t:= \Big(N_j -  \Psi_\infty \Big((\mathbb{R}^4\slash \Gamma_\infty) -  B_\eucl(0, 2t_j^{-\frac{1}{4}}) \Big)\Big)  -  \Big(\bigcup_k \Psi_k(B_\eucl(0,t_k^{\frac{1}{4}}) \Big).
 \end{equation}
On $Y^{16t}\subset Y^t$, we have $g^D = g_\infty$ and on each $N_j^{16t}\subset N_j^t$, we have $g^D = T_j g_{b_j}$. We also define $t_{\max}:= \max_j t_j$. By Definition \ref{def naive desing}, on the intersection $N_j^t\cap Y^t$ we then have $\sqrt{T_j}r_{b_j} = \rad_o $, and on the intersection $N_j^t\cap N_k^t$, we have $\sqrt{T_j}r_{b_j} = \sqrt{T_k}r_{b_k}$.

\begin{defn}[Function $\rad_D$ on a na\"{\i}ve desingularization]\label{definition rD}
    On a na\"{\i}ve desingularization $(M,g^D)$, we define a function $\rad_D$ in the following way:
    \begin{enumerate}
        \item $\rad_D = \rad_o$ on $Y^t$,
        \item $\rad_D = \sqrt{T_j}r_{b_j}$ on each $N_j^t$.
    \end{enumerate}
    The function $\rad_D$ is then smooth on $M$.
\end{defn}

\begin{defn}[Expected topology]
   A  desingularization $(M,g^D_{t,v})$ of a conformally K\"ahler, Einstein orbifold $(X,g_\infty)$ is said to have  the {\em expected topology} if  the Ricci-flat  ALE spaces 
used to construct the    
 desingularization  are all  K\"ahler, in an orientation-compatible manner.
\end{defn}

Notice that the notion of expected topology therefore exactly  models the sort of singularity formation that  can arise from an admissible sequence in the  sense of Definition \ref{ranger}, and so
 in particular  only applies  
to singularities of {\sf type T} in the sense of Definition \ref{loner}. 


\section{Simplest Case of the Proof}
\label{simplest}

We will now illustrate our strategy by proving Theorems \ref{alpha} and \ref{beta} in the simple   case where  $(X,g_\infty )$ only has one singularity, and no bubble-trees arise.

\subsection*{Approximate Einstein-modulo-obstructions metrics}\label{section construction gA}

As was shown in \cite{ozuthese}, around a singular point $p$ of $X$, we can find a coordinate system on $(X,g_\infty )$ in which we have an expansion of  a solution of $\delta_{\eucl}(\Bar{g}_{v_\infty})=0$, for $\eucl$ the Euclidean metric, and
\begin{equation}
        \Bar{g}_{v_\infty} = {\eucl} + \mathfrak{H}_2+ \mathfrak{H}_3+ \mathfrak{H}_4+\mathsf{O}(\rad^5),\label{dvp gov}
\end{equation}
whose Ricci curvature satisfies  
\begin{equation}\label{eq gvo}
    \Ric(\Bar{g}_{v_\infty})-3\Bar{g}_{v_\infty} = \mathbf{w}_{v_\infty} =\mathfrak{w}_0+\mathfrak{w}_2+\mathsf{O}(\rad^2)\in \mathfrak{O}(g_\infty),
\end{equation}
 for a constant $2$-tensor $\mathfrak{w}_0$, and $\mathfrak{w}_2$ quadratic on $\mathbb{R}^4$
\begin{equation}
        \Bar{g}_{v_b} = {\eucl} + \mathfrak{H}^4+ \mathsf{O}(\rad_\eucl^{-8}). \label{dvp gEH}
    \end{equation}
Let $\Sigma:=p^{-1}(\{0\})$ for $p:N\mapsto \mathbb{C}^2/\Gamma$ the minimal resolution if $\Gamma < SU(2)$, or its image in the appropriate  quotient if $\Gamma\not < SU(2)$.

In the rest of the section, we will construct what we call an \textit{approximate obstruction space} (an approximate co-kernel of the linearization of the Einstein equation) and an approximate Einstein-modulo-obstructions desingularization.

\subsubsection*{Extension of obstructions}\label{section extension obstructions}

Let $\mathbf{w} = \mathfrak{w}_0+ \mathfrak{w}_2+ \mathsf{O}(\rad_o^3) \in \mathfrak{O}(g_\infty)$ with $|\mathfrak{w}_i|_{\eucl}\leqslant \rad_o^i$ and $ \mathfrak{o} = O^4 + \mathsf{O}(\rad_o^{-5})\in \mathfrak{O}(\Bar{g}_{v_b})$. 
As proven in \cite[Lemma 2.7]{ozu3}, there is a unique tensor $\mathfrak{w}_0$ on $N$ satisfying: for $\mathfrak{w}_0$ of \eqref{eq gvo},
\begin{empheq}[box=\fbox]{align}
 \phantom{\int}   P_{\bar{g}_{v_b}}w_0&= 0, \nonumber\\
        w_0 &= \mathfrak{w}_0 + \mathsf{O}(\rad^{-4}),
        \label{defw0}   
        \\
        \int_{\Sigma}\langle w_0&, \mho_{v_b}\rangle_{\bar{g}_{v_b}} da_{(\bar{g}_{v_b}|\Sigma)}= 0 \text{ for all  } \mho_{v_b}\in  \mathfrak{O}(\bar{g}_{v_b}). \nonumber
 \phantom{\int}
\end{empheq}
%
%
%
From \cite{biq1,ozu2}, for $\mathfrak{w}_2$ of \eqref{eq gvo}, there exists $w_2$,
\begin{empheq}[box=\fbox]{align}
 \phantom{\int}   P_{\bar{g}_{v_b}}w_2& \in \mathfrak{O}(\bar{g}_{v_b}), \nonumber \\
        w_2 &= \mathfrak{w}_2 + \mathsf{O}(\rad^{-2+\epsilon}),\; \; \forall\; \epsilon>0\\
        \int_{\Sigma}\langle w_2&, \mho_{v_b}\rangle_{\bar{g}_{v_b}} da_{(\bar{g}_{v_b}|\Sigma)}= 0 \text{ for all  } \mho_{v_b}\in  \mathfrak{O}(\bar{g}_{v_b}). \nonumber 
 \phantom{\int}
    \label{def-w2}
\end{empheq}
%
%
%
and a unique symmetric $2$-tensor $\mathfrak{o}^4$ on $X$ such that:
\begin{empheq}[box=\fbox]{align}
 \phantom{\int}     P_{\bar{g}_{v_o}}\mathfrak{o}^4&\in\mathfrak{O}(g_\infty),\nonumber \\
        \mathfrak{o}^4 &= O^4 + \mathsf{O}(\rad_o^{-2-\epsilon}),       \label{def-o2^4} \\
        \mathfrak{o}^4&\perp_{L^2(\bar{g}_{v_o})} \mathfrak{O}(g_\infty).\nonumber 
 \phantom{\int}
\end{empheq}

%
%
%
%
%
%

We then define our extensions of the (co)kernels of the building blocks.

\begin{defn}\label{def extension obst}
Let $\mathbf{w}$ and $\mathfrak{o}$ be as above. We define:
$$\tilde{\mathbf{w}} := \mathbf{w} \#_{p,t} (w_0 + t w_2), \text{ and}$$
$$\tilde{\mathfrak{o}} := (t^2\mathfrak{o}^4)\#_{p,t} \mathfrak{o}.$$
\end{defn}


\subsubsection*{Approximate Einstein-modulo-obstructions metrics}
\label{section first def approx}
    
        For a metric $g$, let us denote $Q_g^{(2)}$ the bilinear terms given by the second derivative of $\mathscr{F}_g$ (whose linearization is $P_g$) at $g$. There exist:
         \begin{enumerate}
            \item  \emph{unique} symmetric $2$-tensors $h_2$ on $N$ and $\mho_2\in \mathfrak{O}(\bar{g}_{v_b})$ satisfying the  equations   
           \begin{empheq}[box=\fbox]{align}
 \phantom{\int}       P_{\bar{g}_{v_b}}h_2&= \lambda \bar{g}_{v_b}+w_0+\mho_{2},\nonumber\\
        h_2 &= \mathfrak{H}_2 + \mathfrak{H}_2^4 + \mathsf{O}(\rad^{-6+\epsilon}), \; \; \forall\; \epsilon>0    \label{def-lambda-k}    \\
        \int_{\Sigma}\langle h_2&, \mho_{v_b}\rangle_{\bar{g}_{v_b}} da_{(\bar{g}_{v_b}|\Sigma)}= 0 \text{ for all  } \mho_{v_b}\in  \mathfrak{O}(\bar{g}_{v_b}). \nonumber 
    \end{empheq}           
        %
%
%
%
           \item a unique symmetric $2$-tensor $h_3$ on $N$ satisfying the following equations
  \begin{empheq}[box=\fbox]{align}
 \phantom{\int}      P_{\bar{g}_{v_b}}h_3&= 0, \nonumber \\
        h_3 &=  \mathfrak{H}_3 +  \mathfrak{H}_3^4 + \mathsf{O}(\rad^{-5+\epsilon}),\; \; \forall\; \epsilon>0
         \label{def-h3}     
        \\
        \int_{\Sigma}\langle h_3&, \mho_{v_b}\rangle_{\bar{g}_{v_b}} da_{(\bar{g}_{v_b}|\Sigma)}= 0 \text{ for all  } \mho_{v_b}\in \mathfrak{O}(\bar{g}_{v_b}).\nonumber 
    \end{empheq}               
           
%
%
%
        
        \item a unique symmetric $2$-tensor $\mathfrak{o}_{(2)}^4$ on $X$ such that:
       \begin{empheq}[box=\fbox]{align}
 \phantom{\int}      P_{\bar{g}_{v_o}}\mathfrak{o}_{(2)}^4&\in\mathfrak{O}(g_\infty), \nonumber\\
        \mathfrak{o}_{(2)}^4 &= O_{(2)}^4 + \mathsf{O}(\rad_o^{-2-\epsilon}),   \label{def-o2^4*}         
        \\
        \mathfrak{o}_{(2)}^4&\perp_{L^2(\bar{g}_{v_o})} \mathfrak{O}(g_\infty), \nonumber 
    \end{empheq}

%
where $\mho_2 =O_{(2)}^4 + \mathsf{O}(\rad_b^{-5})$ in \eqref{def-lambda-k}.

            \item a unique symmetric $2$-tensor $h^4$ on $X$ and $\mathbf{w}^4\in \mathfrak{O}(g_\infty)$ such that:

            \begin{empheq}[box=\fbox]{align}
 \phantom{\int}      P_{\bar{g}_{v_o}}h^4&=\mathfrak{o}_{(2)}^4+\mathbf{w}^4, \nonumber    \\
        h^4 &=  \mathfrak{H}^4 +  \mathfrak{H}^4_2 +  \mathfrak{H}^4_3+  \mathfrak{H}^4_4+ \mathsf{O}(\rad^{1-\epsilon}),   ~   \label{def-h^4}
        \\
        h^4&\perp_{L^2(\bar{g}_{v_o})} \mathfrak{O}(g_\infty).
 \nonumber  
    \end{empheq}


      \item and a symmetric $2$-tensor $h_4$ on $N$ and $\mho_4\in\mathfrak{O}(\bar{g}_{v_b})$ satisfying the following equations
         
            \begin{empheq}[box=\fbox]{align}
 \phantom{\int}      P_{\bar{g}_{v_b}}h_4&= \lambda h_2- Q_{\bar{g}_{v_b}}^{(2)}(h_2,h_2)+w_2+\mho_{4}, \nonumber    \\
       h_4 &=  \mathfrak{H}_4+  \mathfrak{H}^4_4 + \mathsf{O}(\rad^{-4+\epsilon}), ~  \label{def-mu-k}
        \\
       \int_{\Sigma}\langle h_4&, \mho_{v_b}\rangle_{\bar{g}_{v_b}} da_{(\bar{g}_{v_b}|\Sigma)}= 0 \text{ for all  } \mho_{v_b}\in  \mathfrak{O}(\bar{g}_{v_b}).
%
%
%
 \nonumber  
    \end{empheq}

        for all $\epsilon>0$.
\end{enumerate}

\begin{proof}[Proof of the existence of the above tensors]
    The proof of the existence and uniqueness of solutions of \eqref{def-lambda-k}, \eqref{def-h^4} and \eqref{def-mu-k} are exactly as in \cite{ozu3}, which did not depend on any  peculiar features  of the Eguchi-Hanson metric.

    Let us first give a proof for the existence of $h_3$ satisfying \eqref{def-h3}. Here
    the $2$-tensor $\mathfrak{H}_3$ satisfies $P_{\eucl}\mathfrak{H}_3 = 0$ on $\mathbb{R}^4$ so $P_{\bar{g}_{v_b}}(\Upsilon \mathfrak{H}_3) = \mathsf{O}(\rad_b^{-3})$ on $N$ since $\bar{g}_{v_b}-\eucl=\mathsf{O}(\rad_b^{-4})$. From \cite[Lemma 4.3]{ozu2}, it is  possible  to find $h'$ such that $P_{\bar{g}_{v_b}}h'= -P_{\bar{g}_{v_b}}(\Upsilon \mathfrak{H}_3)$ if and only if $ P_{\bar{g}_{v_b}}(\Upsilon \mathfrak{H}_3)\perp_{\bar{g}_{v_b}}\mathfrak{O}(\bar{g}_{v_b})$. Now, by integration by parts: for $\bar{\mho}_{v_b}\in\mathfrak{O}(\bar{g}_{v_b})$,
    $$\int_M\Big\langle P_{\bar{g}_{v_b}}(\Upsilon \mathfrak{H}_3),\bar{\mho}_{v_b}\Big\rangle_{\bar{g}_{v_b}}d\mu_{\bar{g}_{v_b}} = 0.$$
    Indeed, the boundary term is proportional to $ \lim_{r\to\infty} r \int_{\mathbb{S}^3/\Gamma}\langle \mathfrak{H}_3, \mathfrak{H}^4 \rangle_{g_{\mathbb{S}^3/\Gamma}}d\mu_{g_{\mathbb{S}^3/\Gamma}}$, but the coefficients of $\mathfrak{H}_3$ are harmonic in $\rad^3$ while the coefficients of $\mathfrak{H}^4$ are harmonic in $\rad^{-4}$, hence their restrictions on $\mathbb{S}^3/\Gamma$ belong to different eigenspaces of the spherical Laplacian and therefore are orthogonal:
    $$\int_{\mathbb{S}^3/\Gamma}\langle \mathfrak{H}_3, \mathfrak{H}^4 \rangle_{g_{\mathbb{S}^3/\Gamma}}d\mu_{g_{\mathbb{S}^3/\Gamma}}=0.$$
    Here, the existence of $\mathfrak{H}_3^4$ is proved exactly  as in the proof of the existence of $\mathfrak{H}_2^4$ in \cite[Lemma 2.3]{ozu3}.
\end{proof}

We now define the  approximate-solution metrics we will need.

\begin{defn} 
\label{defn approximate einstein mod obst}
    When $v=(v_o,v_b)\in \mathfrak{O}(g_\infty)\times \mathfrak{O}(g_b)$  and $t$ are both sufficiently small, we use the conventions of Definition \ref{def naive desing}
to  define the \textit{approximate Einstein-modulo-obstructions metric} as the na\"{\i}ve desingularization 
  
    $$g^A_{t,v} := (\bar{g}_{v_o} + t^2h^4) \#_{p,t} (\bar{g}_{v_b}+th_2+t^{3/2}h_3+t^2h_4).$$
We  define $\tilde{\mathfrak{O}}(g^A_{t,v})$ as the set of $\tilde{\mathbf{w}} + \Tilde{\mathfrak{o}}$ as in Definition \ref{def extension obst} for 
$\mathbf{w}\in \mathfrak{O}(g_\infty)$ and $\mathfrak{o}\in  \mathfrak{O}(\bar{g}_{v_b})$, 
   and single out a specific  associated approximate  obstruction:
    $$ \mho^A_{t,v} := \tilde{\mathbf{w}}_{v_o}+ t^2\tilde{\mathbf{w}}^4 +t\mho_2+t^2\mho_4\in \tilde{\mathfrak{O}}(g^A_{t,v}).$$
\end{defn}

\subsection*{Control of the approximate solutions}
 The metrics $g^A_{t,v}$ solve  the {\em Einstein-modulo-obstructions} equation in a manner that allows us to 
 naturally extend   the estimate \cite[equation (50)]{ozu3}  as follows: 
 \begin{prop}\label{controle desing un point}
For $0<\beta<\frac{1}{4}$ and $0<\alpha<1$, we have 
\begin{equation}\label{controlled-Einstein-modulo-obstructions}
    \| \mathscr{F}_{g^D_{t,v}}(g^A_{t,v}) + \mho^A_{t,v} \|_{\rad_D^{-2}C^{\alpha}_\beta(g^D_t)} = o(t)
\end{equation}
as $ t\to 0$, while $\|g^A_{t,v}- g^D_{t}\|_{C^{2,\alpha}_{\beta,*}(g^D_{t})} \to 0$ as  $(t,v)\to (0,0)$.
\end{prop}
\begin{proof}
    The proof of this  estimate is very similar to those in \cite{biq1,ozu2,ozu3}. We will focus here on
     the most challenging aspect of the proof, which pertains to the region $t^{\frac{1}{4}}<\rad_o<2t^{\frac{1}{4}}$ or $t^{-\frac{1}{4}}<\rad_b<2t^{-\frac{1}{4}}$
     where the cut-off is used.  Here one has 
    $$ \bar{g}_{v_o} + t^2h^4 = \eucl + \mathfrak{H}_2+\mathfrak{H}_3+\mathfrak{H}_4 +t^2\mathfrak{H}^4 + t^2\mathfrak{H}^4_2 + t^2\mathfrak{H}^4_3+ t^2\mathfrak{H}^4_4 + \mathsf{O}(\rad^5)+ \mathsf{O}(\rad_o^{5-\epsilon}\rad_b^{-4}) $$
    \begin{eqnarray*}
 \phi_{\sqrt{t},*}(\bar{g}_{v_b}+th_2+t^{3/2}h_3+t^2h_4) &=& \eucl + t^2\mathfrak{H}^4 + \mathfrak{H}_2 + t^2 \mathfrak{H}_2^4 + \mathfrak{H}_3 + t^2 \mathfrak{H}^4_3 
 \\&&+ \mathfrak{H}_4 + t^2\mathfrak{H}^4_4 + \mathsf{O}(\rad_o^{2}\rad_b^{-8+\epsilon}) +\mathsf{O}(\rad_b^{-8})\end{eqnarray*}
for all $\epsilon>0$, where $\phi_{\sqrt{t},*}$
 is defined as in Section \ref{reecriture-controle}.
 Thus, in the gluing annuli $t^{\frac{1}{4}}<\rad_o<2t^{\frac{1}{4}}$ and $t^{-\frac{1}{4}}<\rad_b<2t^{-\frac{1}{4}}$ , the difference $$\rad_D^l\left|\nabla_{g^D_t}^l\left((\bar{g}_{v_o} + t^2h^4)-\phi_{\sqrt{t},*}(\bar{g}_{v_b}+th_2+t^{3/2}h_3+t^2h_4)\right)\right|,$$ 
is, for $0<\beta<\frac{1}{4}$,  of uniform size 
$$ \mathsf{O}(\rad_o^{2}\rad_b^{-8+\epsilon}) + \mathsf{O}(\rad_b^{-8}) + \mathsf{O}(\rad^5)+ \mathsf{O}(\rad_o^{5-\epsilon}\rad_b^{-4}) = \mathsf{O}(t^{\frac{5}{4}})=t^{\frac{\beta}{4}}o(t).$$
    Together with the controls on the cut-off function, this shows that \eqref{controlled-Einstein-modulo-obstructions} is satisfied in the annulus $t^{\frac{1}{4}}<\rad_o<2t^{\frac{1}{4}}$ or $t^{-\frac{1}{4}}<\rad_b<2t^{-\frac{1}{4}}$. The proof for the other regions is easier, and can be found in \cite{biq1,ozu3,ozu2}.
\end{proof}

\subsection*{Curvature of the approximate solutions}

Let $x^i$ be canonical coordinate functions on $\mathbb{R}^4$ and $\omega_1^0 = dx^1\wedge dx^2 + dx^3\wedge dx^4$ and $\theta_1^0 = dx^1\wedge dx^2 - dx^3\wedge dx^4$, and similarly $\omega_i^0$ and $\theta_i^0$ for $i\in\{1,2,3\}$. On $g_b$, we denote $\omega_i$ the hyper-K\"ahler forms asymptotic to $\omega_i^0$.

The following proofs rely on the isomorphism between $\Lambda^-_{g_b}\otimes\Lambda^+_{g_b}$ and traceless $2$-tensors obtained as follows: identifying $2$-tensors and endomorphisms of $TM$, the composition of endomorphisms $\omega^-\circ \omega^+$ is symmetric traceless.

\begin{lem}\label{description-W_0}
    Let $\mathfrak{w}_0$ be a constant $2$-tensor on $\mathbb{R}^4$ projected on $\mathbb{R}^4/\Gamma$ for $\Gamma\subset SU(2)$.
 Then, for $i\in \{1,2,3\}$ there exists $\Theta_i\in \Lambda^-_{g_b}$ satisfying $d\Theta_i=0$ such that $\mathfrak{w}_0$ of \eqref{defw0} is equal to $\sum_i \Theta_i\circ \omega_i$.
\end{lem}
\begin{proof}
    Let $\mathfrak{w}_0$ be a $2$-tensor on $g_b$ obtained from \cite[Lemma 2.7]{ozu3}  asymptotic to a constant traceless $2$-tensor $\mathfrak{w}_0 = \sum_{i,j} \mathfrak{w}_{0,ij} \theta_j^0\otimes\omega_i^0$ for $\mathfrak{w}_{0,ij} \in \mathbb{R}$, and the above $\theta_j^0$ and $\omega_i^0$ satisfying $d\theta_i^0=0$ and $d\omega_i^0= 0$. Since $\mathfrak{w}_0$ is traceless, there are anti-self-dual  $2$-forms $\Theta_i$ such that $\mathfrak{w}_0 = \sum_i \Theta_i\circ \omega_i$, $dd^*\Theta_i = 0$, and $\Theta_i = \sum_{j} W_{0,ij} \theta_j^0 +\mathsf{O}(\rad^{-4})$. By integration by parts, this yields:
    \begin{align*}
        0&=\int_M \langle dd^*\Theta_i, \Theta_i\rangle_{g_b}d\mu_{g_b}\\
        &=-\int_M dd^*\Theta_i \wedge \Theta_i\\
        &=-\int_M d(d^*\Theta_i \wedge \Theta_i) + d^*\Theta_i \wedge d\Theta_i\\
        &=-\lim_{\rho\to \infty}\int_{\{\rad_b=\rho\}} d^*\Theta_i \wedge \Theta_i\\
        &\;\;+\int_M |d^*\Theta_i|^2d\mu_{g_b}.
    \end{align*}
    Now, $d^*\Theta_i \wedge \Theta_i = \mathsf{O}(\rad_b^{-5})$ so the boundary term vanishes, and $d^*\Theta_i=0$. Since $\Theta_i$ is anti-self-dual , this implies $d\Theta_i=0$.
\end{proof}

Let us now extend \cite[Lemma 8]{biq2} to the  general setting  in which the obstructions do not all happen to vanish.
\begin{lem}\label{self-dual-curvature}
    Let $g_b$ be an ALE Ricci-flat  K\"ahler  metric, and $h_2$ as in \eqref{def-lambda-k}.
    Then, the first order variation 
    $$\mathcal{R}^{+,(1)}_{g_b}(h_2)$$
    in the direction of $h_2$  of the self-dual  curvature $\mathcal{R}^+$ of $g_b$  has constant eigenvalues equal to those of $\mathcal{R}_{\bar{g}_{v_o}}^+(p)$. 
\end{lem}
\begin{proof}
  Since the assertion is essentially local in nature, we may pass, if necessary,   to a finite  cover of  the relevant ALE space, and thereby arrange for it to be 
   hyper-K\"ahler,  as in the proof of \cite[Corollary 6.4]{ozu2}.
    
    The proof of \cite[Lemma 8]{biq2} using \cite[Lemma 3]{biq2} now immediately extends  to show that $\mathring{\Ric}^{(1)}_{g_b}(h_2) = \sum_{i}\Omega_i\circ \omega_i$ for $\Omega_i$ anti-self-dual  $2$-forms satisfying $d\Omega_i=0$ and asymptotically constant (or decaying at infinity), and $\omega_i$ the three K\"ahler forms on $g_b$. This is exactly our situation by Lemma \ref{description-W_0}. 
\end{proof}

We can therefore precise our estimate of the curvature of these metrics: 
\begin{cor}\label{curvature-metrics}
    The self-dual Weyl curvature $W^+$ of $g^A_{t,v}$ at any point  has eigenvalues arbitrarily close to those of $W^+(g_\infty)$  at some corresponding  point of $X$, up to an error that tends to zero uniformly as $(t,v)\to 0$.
\end{cor}
\begin{proof}
On the region $\rad_o>2t^{\frac{1}{4}}$ where $g^{A}_{t,v} = \bar{g}_{v_o} + t^2h^4$, we have $|\nabla^l_{g_\infty}(\bar{g}_{v_o}-g_\infty)| \leqslant C \|v_o\|_{C^2}$, and $t^2|\nabla^l_{g_\infty}h^4|\leqslant t^{1-\frac{l}{4}}$. Consequently, when $(t,v_o)\to 0$, we have $ |\nabla^l_{g_\infty}(\bar{g}_{v_o} + t^2h^4-g_\infty)| \to 0$ for $ l<4 $, and in particular, the self-dual Weyl curvature of $\bar{g}_{v_o} + t^2h^4$ uniformly tends to that of $g_\infty$ in this region. 

For $\delta>0$ small enough independent of $(t,v)$, on the region $\delta^{-1}<\rad_b<t^{-1/4}$, where $\frac{1}{t}g^{A}_{t,v} =\bar{g}_{v_b} +th_2+t^{\frac{3}{2}}h_3+t^2h_4$, we have $$g^{A}_{t,v} - \bar{g}_{v_o} = t^2 (\mathfrak{H}^4+ \mathfrak{H}_2^4+\mathfrak{H}_3^4+\mathsf{O}(\rad_b^{-8})),$$
and since $\mathcal{R}^{+,(1)}_{\eucl}(\mathfrak{H}^4) = 0$, this yields 
$$ \mathcal{R}^{+}(g^{A}_{t,v}) = \mathcal{R}^{+}(\bar{g}_{v_o}) + \mathsf{O}(t^2 \rad_b^{-6} ) = \mathsf{O}(t^{\frac{1}{2}})=o(1),$$
and hence 
$$W^{+}(g^{A}_{t,v}) = W^{+}(\bar{g}_{v_o}) + \mathsf{O}(t^2 \rad_b^{-6} ) = \mathsf{O}(t^{\frac{1}{2}})=o(1).$$

Finally, on the region $\rad_b<\delta^{-1}$, where $\frac{1}{t}g^{A}_{t,v} =\bar{g}_{v_b} +th_2+t^{\frac{3}{2}}h_3+t^2h_4$, for $l \leqslant 3$, we have
$t^{\frac{i}{2}}|\nabla^l_{\bar{g}_{v_b}}h_i|\leqslant C t^{\frac{i}{2}}$ for $C$ independent of $t$ and $v$ small enough. We therefore have 
$$ W^{+}(g^{A}_{t,v}) =  W^{+,(1)}_{\bar{g}_{v_b}}(h_2) + \mathsf{O}(t). $$

However, because the Einstein orbifold $(X,g_\infty )$ is assumed to be limit of an admissible sequence of $\lambda > 0$ Einstein metrics, $\mathcal{R}^{+}(g_\infty )= \frac{s}{12} g_\infty+ W^+$ must
have zero as an eigenvalue at the singular point $p$ by \cite{ozu2}, which refined  an earlier result of  \cite{biq2}. However,  since Theorem \ref{hemhaw} tells us $g_\infty$ is conformal 
to a K\"ahler metric of positive scalar curvature,   the two smallest eigenvalues of the trace-free tensor $W^+$ are everywhere equal, it therefore follows that 
$\mathcal{R}^{+}(g_\infty )$ has eigenvalues $(0,0,\frac{s}{4})=(0,0,\lambda)$ at the singular point $p$. By Lemma \ref{self-dual-curvature}, it therefore follows that 
$\mathcal{R}^{+,(1)}_{\bar{g}_{v_b}}(h_2)$ has eigenvalues arbitrarily close to $(0,0,\lambda )$ as $(t,v)\to 0$. The eigenvalues of $W^+$ therefore tend to $(-\frac{\lambda}{3}, -\frac{\lambda}{3}, \frac{2\lambda}{3})$ 
in this region, and the claim therefore holds everywhere.  \end{proof}

\subsection*{Proof-of-Concept  Versions of Theorems \ref{alpha} and \ref{beta}}\label{no tree one sing}

Now suppose  that  $(X,g_\infty )$  is a conformally K\"ahler, Einstein orbifold with  $\lambda =3$, with {\sf exactly one} singularity, for which there is an admissible sequence $(M,g_i)$
with  $$d_{GH}\Big((M,g_i),(X,g_\infty)\Big)\to 0.$$
Furthermore,  assume  that any Ricci-flat ALE rescaled pointed limit that bubbles off from the sequence  has no orbifold singularities.

By \cite{ozu2}, then, after possibly passing to  a subsequence, there then exist scales $t_i$,  na\"{\i}ve desingularizations $g^D_{t_i}$, and parameters $v_i\in \mathfrak{O}(g_b)\times \mathfrak{O}(g_\infty)$ such that
\begin{itemize}
    \item $\|g_i-g^A_{t_i,v_i}\|_{C^{2,\alpha}_{\beta,*}(g^D_{t_i})} \to 0$, and
    \item $g_i-g^A_{t_i,v_i} \perp_{g^D_{tj}}\Tilde{\mathfrak{O}}(g^{A}_{t_i,v_i})$.
\end{itemize}
Then  \cite[Proposition 5.11]{ozu2} tells us that  
$$\|g_i-g^A_{t_i,v_i}\|_{C^{2,\alpha}_{\beta,*}(g^D_{t_i})} \leqslant C \|\mathscr{F}_{g^D_{t_i}}(g^A_{t_i,v_i}) + \Tilde{\mathfrak{O}}^A_{t_i,v_i}\|_{\rad_D^{-2}C^\alpha_\beta(g^D_{t_i})}=o(t_i)$$
as $i\to \infty$, by virtue  of   \eqref{controlled-Einstein-modulo-obstructions}. This then allows us to  obtain the following curvature-control: for some $C>0$, independent of $i$,
\begin{equation}\label{estimate W+} \hspace{-.008in}
    \| \operatorname{W}^+(g^A_{t_i,v_i}) - \operatorname{W}^+(g_i) \|_{C^{0}(g_i)}\leqslant C t_i^{-1} \| \operatorname{W}^+(g^A_{t_i,v_i}) - \operatorname{W}^+(g_i) \|_{\rad_D^{-2}C^{\alpha}_\beta(g^D_{t_i})} \to 0
\end{equation}
as $i\to \infty$, 
since, by definition of the norms: $\|\cdot \|_{C^{0}(g_i)} \;\leqslant \; C t_i^{-1} \|\cdot\|_{\rad_D^{-2}C^{\alpha}_\beta(g^D_{t_i})}$ because $\rad_D^2>t_i$.
However,  because $\det (W^+) (g_\infty) > 0$ everywhere,Corollary \ref{curvature-metrics} implies 
  that $\det\operatorname{W}^+(g^A_{t_i,v_i})>0$ for $i \gg 0$. By  \eqref{estimate W+}, the {open} property 
$\det \operatorname{W}^+>0$  is therefore inherited by the genuine Einstein metrics $g_i$ for sufficiently large $i$. We have thus shown  that 
the   Einstein metrics $g_i$  satisfy   Wu's  criterion  \eqref{pengwu},  so it  follows  \cite{lebdet,pengwu}  that the $g_i$ must be  \textit{conformally K\"ahler}
 when $i\gg 0$. However, such Einstein metrics on smooth $4$-manifolds are actually all 
 K\"ahler-Einstein, aside  \cite[Theorem A]{lebuniq}  from  two strongly-rigid exceptions   for which,  up to isometries and rescalings, there is only one conformally-K\"ahler Einstein manifold on the given $4$-manifold; but this rigid scenario 
 is then precluded by our assumption that the limit has an orbifold singularity. 
 We have thus shown that Theorems \ref{alpha} and \ref{beta} do in fact  hold in the test-case where $X$ has exactly  one singularity,  and no trees of singularities occur. 
 \hfill $\Box$

\section{Without  Trees of Singularities}
\label{indicative}

We now allow multiple singularities, but  exclude the formation of 
 trees of singularities.  For example, as explained  in Corollary \ref{eagle} below, 
 this  applies  to  desingularizations of  a  K\"ahler-Einstein orbifolds with $c_1^2 =4$. 
In general terms, this case is very similar to the situation discussed in \S \ref{simplest}, and can best be understood using the framework  of  \cite[Section 5.2]{ozu2} through the use 
of \textit{partial} Einstein-modulo-obstructions desingularizations.
Roughly speaking, the idea is to desingularize the singularities one-by-one, from the smallest scale to the largest one. {\em A priori}, one can only partially desingularize by Einstein-modulo-obstruction metrics. Starting by the smallest scale is crucial as it ensures negligible perturbations of the orbifold at its remaining singular points.
\\

Let us once again take $(M,g_i)$ to be  an admissible  sequence of Einstein metrics that converges to a conformally K\"ahler, Einstein orbifold $(X,g_\infty)$  in the Gromov-Hausdorff sense.
We  now allow $X$ to possibly have  \textit{multiple} singularities, but we still suppose  trees of singularities do not occur. As in the previous Section \ref{no tree one sing}, by \cite{ozu1,ozu2}, then, up to taking a subsequence, for $i$ large enough, there exists a na\"{\i}ve desingularization $g^D_{t_i}$ and parameters $v_i\in \Tilde{\mathfrak{O}}(g^D_{t_i})$ such that:
 $\|g_i-g^D_{t_i}\|_{C^{2,\alpha}_{\beta,*}(g^D_{t_i})} \to 0$.
As in \cite[Section 5.2]{ozu2}, we will produce a better approximate metrics $g^A_{t_i,v_i}$ by using \textit{partial} Einstein-modulo-obstructions desingularizations as stepping stones.

\subsection*{Iterative partial na\"{\i}ve  desingularizations}

After passing to a subsequence, we may assume  that we can order the $n$ singular points of $X$ as $p_1,...,p_n$ so that
$$0<t_{1,i}\leqslant t_{2,i} \leqslant... \leqslant t_{n,i},$$
where $t_{k,i}$ is the scale of the gluing at $p_k$ in the construction of $g^A_{t_i}$ in Definition \ref{defn approximate einstein mod obst}.

\begin{defn}[Partial na\"{\i}ve desingularizations]
  For $k\in \{1,...,n\}$, the $k^{\rm th}$ \emph{partial} desingularization of $(X,g_\infty )$, which we will denote by $(Y_k, g^{D_k}_{t})$,  is obtained by the  na\"{\i}ve gluing of $g_{b_1}$, ..., $g_{b_k}$ to  $(X,g_\infty )$ at the corresponding points $ p_1 $,..., $p_k$ at the scales $t_1$,..., $t_k$.
\end{defn}

\noindent By iteration on $k$, we can then define  approximate-Einstein-modulo-obstructions metrics $g^{A_k}_{t,v}$ and {\sf actual} Einstein-modulo-obstructions metrics $\hat{g}^{k}_{t,v}$ on $Y_k$ as follows:

\paragraph{For $k=0$,} we define $g^{D_0}_{t,v}=g^{A_0}_{t,v} = \hat{g}^{0}_{t,v} = \bar{g}_{v_o}$ and $\tilde{\mathfrak{O}}(g^{A_0}_{t,v}):=\mathfrak{O}(g_\infty)$.

\paragraph{Now let $0\leqslant k \leqslant n-1$.} 

The idea is to replace $\bar{g}_{v_o}$ by $\hat{g}^{k}_{t,v}$ and $\mathfrak{O}(g_\infty)$ by $\tilde{\mathfrak{O}}(g^{A_k}_{t,v})$ and to perform the constructions as in Section \ref{simplest}.

We start with some $\hat{g}^{k}_{t,v}$ that  in a neighborhood of $p_{k+1}$  satisfies $$\Ric(\hat{g}^{k}_{t,v})-3\hat{g}^{k}_{t,v} = w^k_{0}+w^k_{2}+\mathsf{O}(\rad_o^3)\in \mathfrak{O}(g_\infty),$$
and has an expansion 
\begin{equation}
    \hat{g}^{k}_{t,v} = {\eucl} + \mathfrak{H}_2(k)+\mathfrak{H}_3(k)+\mathfrak{H}_4(k)+\mathsf{O}(\rad^5),\label{dvp gktv}
\end{equation}
exactly like \eqref{dvp gov}. We can therefore define the approximate Einstein-modulo-obstructions $g^{A_{k+1}}_{t,v}$ as in Section \ref{section construction gA} as the desingularization of $\hat{g}^{k}_{t,v}$ at $p_{k+1}$ by $\bar{g}_{v_b}$ at scale $t_{k+1}$ replacing the $\mathfrak{H}_i$ by the present $\mathfrak{H}_i(k)$. Indeed, all of the steps of the construction are local and do not see the global properties of $g_\infty$, \textit{except} the construction of $ h^4 $ in \eqref{def-h^4} with $\mho(g_\infty)$ by $\tilde{\mho}(g^{A_k}_{t,v})$, which follows from the application of \cite[Proposition 4.9]{ozu2}  to $P_{g^{D_k}_{t,v}}(\Upsilon(\mathfrak{H}^4+\mathfrak{H}^4_2+\mathfrak{H}^4_3))$, where $\Upsilon$ is a controlled cut-off function supported in a neighborhood of $p_{k+1}$.
\\

We also extend the definition of $\tilde{\mathfrak{O}}(g^{A_{k+1}}_{t,v})$ as in Definition \ref{defn approximate einstein mod obst}: as the set of $\tilde{\mathbf{w}}+\tilde{\mathfrak{o}}$ for $(\mathbf{w},\mathfrak{o})\in \tilde{\mathfrak{O}}(g^{A_{k}}_{t,v})\times \mathfrak{O}(\bar{g}_{v_b})$, defined exactly as in Section \ref{section extension obstructions} replacing $\Bar{g}_{v_o}$ by $g^{A_{k}}_{t,v}$ and $\mathfrak{O}(g_\infty)$ by $\tilde{\mathfrak{O}}(g^{A_{k}}_{t,v})$.

Now, by \cite[Theorem 4.10]{ozu2}, there exists a unique Einstein-modulo-obstructions metric $\hat{g}^{k+1}_{t,v}$ satisfying: for $(t,v)$ small enough,
\begin{itemize}
    \item $\|\hat{g}^{k+1}_{t,v}-g^{A_{k+1}}_{t,v}\|_{C^{2,\alpha}_{\beta,*}(g^{D_k}_{t})} \to 0$ as $(t,v)\to 0$,
    \item $\hat{g}^{k+1}_{t,v}-g^{A_{k+1}}_{t,v}\perp_{g^{D_k}_{t,v}} \tilde{\mathfrak{O}}(g^{A_{k+1}}_{t,v}) $, and
    \item $\mathscr{F}_{g^{D_{k+1}}_{t}}(\hat{g}^{k+1}_{t,v}) \in \tilde{\mathfrak{O}}(g^{A_{k+1}}_{t,v})$.
\end{itemize}

\subsection*{Controlling the approximate desingularizations}

As in the proof of \cite[(96), Proposition 5.13]{ozu2}, the proof of Proposition \ref{controle desing un point}, which is local in nature, extends to the case of partial desingularizations.
\begin{prop}\label{controle desing partielle}
    Let $k\in\{1,...,n\}$. We have the following controls:
    \begin{enumerate}[{\rm (i)}]
        \item $\|\mathscr{F}_{g^{D_k}_{t}}(g^{A_k}_{t,v}) - \mho^{A_k}_{t,v}\|_{\rad_D^{-2}C^{\alpha}_\beta(g^{D_k}_{t})}=o(t_k), \text{ as $t_k\to 0$},$ \label{one}
        \item $\|\hat{g}^{k}_{t,v}-g^{A_k}_{t,v}\|_{C^{2,\alpha}_{\beta,*}(g^{D_k}_{t})}=o(t_k), \text{ as $t_k\to 0$},$ and \label{two} 
        \item at every point, the eigenvalues of the self-dual  Weyl curvature of $\hat{g}^{k}_{t,v}$ are  arbitrarily close to those of $g_\infty$ at some corresponding point of $X$. 
         \label{three} 
    \end{enumerate}
\end{prop}
\begin{proof}
    We will prove this by induction on $k$.
   When  $k=0$, one has $g^{D_0}_{t,v}=g^{A_0}_{t,v} = \hat{g}^{0}_{t,v} = \bar{g}_{v_o}$ and the result is clear.
      Proceeding by induction, let us now assume that the Proposition holds for $0\leqslant k \leqslant n-1$.
    
 To prove \eqref{one}, notice that    the control of $$\|\mathscr{F}_{g^{D_{k+1}}_{t}}(g^{A_{k+1}}_{t,v}) - \mho^{A_{k+1}}_{t,v}\|_{\rad_D^{-2}C^{2,\alpha}_\beta(g^{D_{k+1}}_{t})}$$ is exactly as in the proof of \eqref{controlled-Einstein-modulo-obstructions}, as the argument  is local, and therefore also applies to partial desingularizations; cf.  \cite[Proposition 5.13]{ozu2}.

To prove \eqref{two}, notice that, as in 
Section \ref{no tree one sing}, we once again have 
$$\|\hat{g}^{{k+1}}_{t,v}-g^{A_{k+1}}_{t,v}\|_{C^{2,\alpha}_{\beta,*}(g^{D_{k+1}}_{t,v})}\leqslant C\|\mathscr{F}_{g^{D_{k+1}}_{t}}(g^{A_{k+1}}_{t,v}) - \mho^{A_{k+1}}_{t,v}\|_{\rad_D^{-2}C^{\alpha}_\beta(g^{D_{k+1}}_{t,v})}$$
for some $C> 0$,  where the right-hand side is $o(t_{k+1})$ as $t_{k+1}\to 0$.  
Indeed, this follows from \eqref{one} and \cite[(88), Proposition 5.11]{ozu2}, the  proof of which 
does not require the existence of an  Einstein desingularization, but instead holds for any {\em Einstein-modulo-obstructions} metric.

 Finally, to prove \eqref{three}, notice that the proof of Corollary \ref{curvature-metrics} also works in our present setting, both at  and away from the orbifold singularities. 
    By the property for $k$, the metric $\hat{g}^k_{t,v}$, which plays the role of $\bar{g}_{v_o}$ in Section \ref{no tree one sing}, has self-dual Weyl curvature with eigenvalues arbitrarily close  to those
    of $(X,g_\infty )$ at some corresponding. An extension of the proof of Corollary \ref{curvature-metrics}, which is local in nature, then  shows that the eigenvalues of $W^+(g^{A_{k+1}}_{t,v})$ are arbitrarily close to
    those of  $W^+(g_\infty )$ at an appropriate point.
    We then conclude, in  exact analogy to  \eqref{estimate W+}, that 
    \begin{eqnarray}\label{estimate W+ partiel}
    \| \operatorname{W}^+(g^{A_{k+1}}_{t,v}) &-& \operatorname{W}^+(\hat{g}^{k+1}_{t,v}) \|_{C^{0}(g^{D_{k+1}}_t)} \\
    &\leqslant&C t^{-1} \| \operatorname{W}^+(g^{A_{k+1}}_{t,v}) - \operatorname{W}^+(\hat{g}^{k+1}_{t,v}) \|_{\rad_D^{-2}C^{\alpha}_\beta(g^{D_{k+1}}_t)} \nonumber 
\end{eqnarray}
where the right-hand side tends to zero as $(t,v)\to 0$.
\end{proof}

\subsection*{Conclusion of the proof}

Bringing these ideas together, let us now  assume  
\begin{itemize}
\item that  $X^4$ is an orbifold with   $n$ isolated singularities $p_1,...,p_n$;
    \item that  $(X^4,g_\infty )$  is Hermitian and Einstein, with $\lambda >0$;
    \item that there is an admissible  sequence $(M,g_i)$ with 
    $$d_{GH}\Big((M,g_i),(X,g_\infty)\Big)\to 0; \quad \mbox{and}$$
    \item that no trees of singularities arise  in this limiting process. 
\end{itemize}

By \cite{ozu2}, then, up to taking a subsequence, for $i$ large enough, there exist $t_i>0$, along with na\"{\i}ve desingularizations $g^D_{t_i}$ and parameters $$v_i\in \mathfrak{O}(g_{b_1}) \times ... \times \mathfrak{O}(g_{b_n}) \times \mathfrak{O}(g_\infty)$$ for which 
\begin{itemize}
    \item $\|g_i-g^{A_n}_{t_i,v_i}\|_{C^{2,\alpha}_{\beta,*}(g^D_{t_i})} \to 0$, and
    \item $g_i-g^{A_n}_{t_i,v_i} \perp_{g^D_{t_i}}\Tilde{\mathfrak{O}}(g^{A_n}_{t_i,v_i})$.
\end{itemize}
Since our assumptions in particular imply that $M$ has the {expected topology}, and all of the Ricci-flat ALE spaces arising from curvature concentration are K\"ahler, 
our assumption  that {no trees of singularities arise}
then implies \cite[\S 7.1]{ozu2}
that, after passing to  a subsequence and up to a diffeomorphism, $g_i= \hat{g}^n_{t_i,v_i}$ for one of the above $\hat{g}^n_{t_i,v_i}$. 
In particular, $g_i$ is an Einstein metric  
satisfying Wu's criterion $\det\operatorname{W}^+(g_i)>0$, and it 
 therefore follows \cite{lebdet,pengwu}  that, for $i$ large enough, the  $g_i$ are necessarily  \textit{conformally K\"ahler}. However, as previously noted, such metrics on a smooth $4$-manifold $M$  are actually 
 K\"ahler-Einstein unless \cite{lebuniq} there is,  up to isometries and rescalings,  only one conformally-K\"ahler Einstein manifold on $M$. 
Theorems \ref{alpha} and \ref{beta} thus now follow in the  case where  no trees of singularities arise.  \hfill $\Box$

\medskip 

\noindent Combining this with Proposition  \ref{birdy}
now yields  the following application:

\begin{cor} \label{eagle}
Let $(X^4,g_\infty)$  be a K\"ahler-Einstein orbifold with at least one singular point. If  $c_1^2 (X) \geq 4$, and if $(X, g_\infty)$ is   the Gromov-Hausdorff limit of an 
admissible sequence $(M, g_i)$ of smooth $\lambda > 0$ Einstein $4$-manifolds, then $c_1^2 (X) =4$, and $M\approx \CP_2 \# 5 \overline{\CP}_2$.
\end{cor} 

\begin{proof} When  $c_1^2 (X)\geq 4$ and $(X,g_\infty )$ is K\"ahler-Einstein, 
Proposition  \ref{birdy} tells us that each singularity of $X$ is modeled on $\CC^2/\ZZ_2$, and this in particular
excludes the formation of bubble-trees. The above 
 argument therefore shows that the Einstein $4$-manifolds $(M,g_i)$ must be Hermitian 
for all sufficiently large $i$.   However,  classification \cite{lebuniq} tells us 
that, up to isometry and rescalings,  there is at most one $\lambda >0$ Hermitian, Einstein metric on the underlying smooth $4$-manifold $M$  of a complex surface with $c_1^2 \geq 5$. When  $c_1^2(X)\geq 4$,  singularities can therefore
only form in the limit
when  $c_1^2 (M) = c_1^2 (X) =4$ and    $M$ is 
  diffeomorphic to  $\CP_2 \# 5 \overline{\CP}_2$.
\end{proof}

\section{The General Case}
\label{forest} 

Finally, we deal with trees of singularities by further refining our notion of approximate Einstein-modulo-obstructions metric.
In fact, our proof of Proposition \ref{once} already depended on the reason why it is  possible  to handle trees of  Ricci-flat K\"ahler ALE metrics: when glued together in the expected manner, they can always
be approximated by a {\em single}  Ricci-flat K\"ahler ALE metric \cite{ozu2}. The essence of what we will now do   amounts to a controlled, quantitative implementation of  this observation. While
 there are more elementary ways of carrying this out for  the Gibbons-Hawking gravitational instantons, our proof will be geared to handle the general case.

\subsubsection*{Asymptotics of the building blocks}

We will now consider:
\begin{itemize}
    \item $(N_1,\bar{g}_{v_{b_1}})$ a Ricci-flat ALE orbifold with a singularity at $p_1$, and
    \item $(N_2,\bar{g}_{v_{b_2}})$ glued at scale $t_2$ at a remaining singular point $p_1$ of $(N_1,\bar{g}_{v_{b_1}})$.
\end{itemize}

We have the following expansions of the different metrics: for any $l\leqslant 4$,
\begin{itemize}
    \item at its singular point $p_1$: $\bar{g}_{v_{b_1}}= \eucl + \mathfrak{H}_2+\mathfrak{H}_3+\mathfrak{H}_4+... +\mathfrak{H}_l + \mathsf{O}(\yha_{b_1}^{l+1})$
    \item at its infinity: $\bar{g}_{v_{b_2}} = \eucl + \mathfrak{H}^4+...+\mathfrak{H}^l+ \mathsf{O}(\yha_{b_2}^{l+1}).$
\end{itemize}
In particular, developing the equations $\mathscr{F}_\eucl(g_{b_1})=0$ and $\mathscr{F}_\eucl(g_{b_2})=0$ according to powers of $r_{b_1}$ and $r_{b_2}$ the above tensors satisfy equalities:
\begin{itemize}
    \item $P_\eucl\mathfrak{H}_i + Q^{[i],+}_\eucl(\mathfrak{H}_2,\mathfrak{H}_3,...,\mathfrak{H}_{i-1}) = 0 $, and
    \item $P_\eucl\mathfrak{H}^j + Q^{[j],-}_\eucl(\mathfrak{H}^4,\mathfrak{H}^5,...,\mathfrak{H}^{j-1}) = 0 $,
\end{itemize}
where $Q^{[j],\pm}_\eucl$ are specific finite sums of multilinear operations on the tensors and their two first derivatives. The key point is that $Q^{[i],+}_\eucl(\mathfrak{H}_2,\mathfrak{H}_3, \ldots ,\mathfrak{H}_{i-1}) = \mathsf{O}(\rad^{i-2})$, and $Q^{[j],-}_\eucl(\mathfrak{H}^4,\mathfrak{H}^5, \ldots ,\mathfrak{H}^{j-1})=\mathsf{O}(\rad^{-j-2})$.

\subsubsection*{Approximations of  
Ricci-flat K\"ahler metrics}

   When $v_b = (v_{b_1},t_2,v_{b_2})$ is sufficiently small, we will use  $g^B_{v_b}$ to denote 
   the na\"{\i}ve gluing $\Bar{g}_{v_{b_1}}\#_{p_1,t_2}\bar{g}_{v_{b_2}}$ and let $\rad_b$ be the associated extension of $\rad_D$ of Definition \ref{definition rD}. By \cite[Section 6.2]{ozu2}, there is a unique  Ricci-flat K\"ahler ALE metric $\bar{g}_{v_b}$ that,  for some $\epsilon>0$, satisfies
       \begin{enumerate}
        \item $\mathscr{F}_{g^B_{v_b}}(\Bar{g}_{v_b}) =0,$
        \item $\|\Bar{g}_{v_b}-g^B_{v_b}\|_{C^{2,\alpha}_{\beta,**}}\leqslant \epsilon$, and
        \item $\Bar{g}_{v_b}-g^B_{v_b}\perp_{L^2(g_b)}\Tilde{\mathfrak{O}}(g^B_{v_b})$,
    \end{enumerate}
    where $C^{2,\alpha}_{\beta,**}(g^B_{v_b})$ is defined in  Definition \ref{norm orbifold ALE} below. 
       We emphasize that  this construction  reaches \textit{all} nearby Ricci-flat ALE metrics.

The results of \cite[Section 6.2]{ozu2} tell us that any Einstein-modulo-obstructions perturbation of a tree of  Ricci-flat K\"ahler ALE metrics is in fact  a  Ricci-flat K\"ahler ALE metric. For any $l\in \mathbb{N}$, this metric can then be approximated by a na\"{\i}ve gluing: 
\begin{equation}\label{eq def gBL}
    g^{B,[l]}_{v_b}:=\Big(\overline{g}_{v_{b_1}}+\sum_{j=4}^{l}t_2^{\frac{j}{2}}h^j\Big) \#_{p_1,t_2}\Big(\overline{g}_{v_{b_2}}+\sum_{i=2}^{l}t_2^{\frac{i}{2}}h_i\Big),
\end{equation}
    where for $i,j\leqslant 4$, one defines the $h_i$ and $h^j$ as in \eqref{def-lambda-k}, \eqref{def-h3}, \eqref{def-h^4} where all of the obstructions $\mho_i$ and $\mathbf{w}^4$ vanish. 
  Following   \eqref{def-mu-k}, we now solve iteratively for $i,j\geqslant 4$ by applying the following two steps: 
    \begin{enumerate}
        \item first, solve the system 
            \begin{empheq}[box=\fbox]{align}
    P_{\bar{g}_{v_{b_1}}}h^j&= -Q_{\bar{g}_{v_{b_1}}}^{[j],-}(h^4,...,h^{j-1}),\nonumber \\
        h^j &= \sum_{i\leqslant j-1} \mathfrak{H}_i^j + \mathfrak{H}_j^j + \mathsf{O}(\rad^{-1-\epsilon}),\; \; \forall\; \epsilon>0,  \label{def-h^j} \\
        \int_{\Sigma}\langle h^j&, \mho_{v_{b_1}}\rangle_{\bar{g}_{v_{b_1}}} da_{(\bar{g}_{v_{b_1}|\Sigma})}= 0 \quad \forall  \mho_{v_{b_1}}\in  \mathfrak{O}(\bar{g}_{v_{b_1}})
 \nonumber  
    \end{empheq}

%
%

        where $Q_{\bar{g}_{v_{b_1}}}^{[j],-}$ involves the nonlinear terms in $h^l$ for $l\leqslant j-1$ appearing when solving $\mathscr{F}_{\bar{g}_{v_{b_1}}}(g)=0$, and where the $\mathfrak{H}_i^j$ for $i\leqslant j-1$ are determined by the previously defined $h_i$. In order to solve this equation, notice that for $\tilde{\Upsilon}$ a cut-off function supported in a neighborhood of $p_1$
        $$ P_{\bar{g}_{v_{b_1}}}\Big(\Tilde{\Upsilon}\sum_{i\leqslant j-1} \mathfrak{H}_i^j\Big) -Q_{\bar{g}_{v_{b_1}}}^{[j],-}(h^4,...,h^{j-1}) = \mathsf{O}(r_{b_1}^{-2-\epsilon}), \;\; \forall \epsilon>0, $$
        hence, by \cite[Lemma 4.42]{ozuthese}, one can find $h'$ such that 
    $$P_{\bar{g}_{v_{b_1}}}\Big(h'+\Tilde{\Upsilon}\sum_{i\leqslant j-1} \mathfrak{H}_i^j\Big) -Q_{\bar{g}_{v_{b_1}}}^{[j],-}(h^4,...,h^{j-1})\in \mathfrak{O}(\Bar{g}_{v_{b_1}}).$$
        Since the gluing is unobstructed by \cite[Section 6.2]{ozu2}, one then shows iteratively that the element of $\mathfrak{O}(\Bar{g}_{v_{b_1}})$ vanishes, and $h' = \mathfrak{H}_j^j + \mathsf{O}(\rad^{-1-\epsilon})$ for $\mathfrak{H}_j^j = \mathsf{O}(\rad^{-\epsilon})$ for all $\epsilon>0$, and in fact, for any $L\geqslant 2$ and $\epsilon>0$, there is a neighborhood of infinity where
        $$ h' = \mathfrak{H}_j^j+ \sum_{l= j+1}^L \mathfrak{H}^j_l + \mathsf{O}(\rad_o^{L-j-\epsilon}), $$
        \item then, one similarly solves the equation 
            \begin{empheq}[box=\fbox]{align}
    P_{\bar{g}_{v_{b_2}}}h_j&= -Q_{\bar{g}_{v_{b_2}}}^{[j],+}(h_2,...,h_{j-1}),    \nonumber   \\
        h_j &= \sum_{i\leqslant j} \mathfrak{H}_j^i + \mathfrak{H}_j^{j+1} + \mathsf{O}(\rad^{-2+\epsilon}),\; \; \forall\; \epsilon>0,  \label{def-h_j}\\
        \int_{\Sigma}\langle h_j&, \mho_{v_{b_2}}\rangle_{\bar{g}_{v_{b_2}}} da_{(\bar{g}_{v_{b_2}|\Sigma})}= 0 \text{ for all  } \mho_{v_{b_2}}\in  \mathfrak{O}(\bar{g}_{v_{b_2}}).
 \nonumber  
    \end{empheq}

%
%
%

        where $Q_{\bar{g}_{v_{b_2}}}^{[j],+}$ involves the nonlinear terms in $h^l$ for $l\leqslant j$ appearing when solving $\mathscr{F}_{\bar{g}_{v_{b_2}}}(g)=0$, and where the $\mathfrak{H}_j^i$ for $i\leqslant j$ are determined by the previous $h^l$ defined. One then obtains $ \mathfrak{H}^{j+1}_j $ determined from the unique solution of \eqref{def-h_j} such that in a neighborhood of infinity, for all $L\geqslant 4$, and $\epsilon>0$ $h_j = \sum_{l=4}^L \mathfrak{H}_j^l + \mathsf{O}(\rad_b^{j-L+\epsilon})$.
    \end{enumerate}

The point is that there are \textit{no obstructions} \cite[Section 6.2]{ozu2} which lets us estimate the metrics directly without having to approximate obstructions. Explicit computations generalizing slightly those of Proposition \ref{controlled-Einstein-modulo-obstructions} (see also \cite[Sections 2 and 3]{ozu3}) lead to the following result.
\begin{prop}\label{prop controle KRFALE}
    For any $l\geqslant 4$ and $0<\alpha<1$ and $0<\beta<\frac{1}{4}$, there is $C>0$ independent of $v_b$ so that we have the control 
    $$\left\|(1+\rad_b)^4\mathscr{F}_{g_{v_b}^B}(g^{B,[l]}_{v_b})\right\|_{\rad_b^{-2}C^{\alpha}_\beta(g^{B}_{v_b})}\leqslant C t_2^{\frac{l+1-\beta}{4}} = o(t_2^{\frac{l}{4}})$$
    and consequently, by {\em\cite[(5.15), Lemma 5.19]{ozuthese}}, we therefore have
    $$\|\Bar{g}_{v_b}-g^{B,[l]}_{v_b}\|_{C^{2,\alpha}_{\beta,**}(g^{B}_{v_b})}\leqslant C't_2^{\frac{l+1-\beta}{4}} = o(t_2^{\frac{l}{4}})$$
      for some $C'>0$ independent of $v_b$,
\end{prop}

In particular, considering $l\geqslant 9$ for $g_b$, we have the following controls.
\begin{cor}\label{controls ALE bar gvb}
We have the following more concrete controls on $\Bar{g}_{v_b}$:
    \begin{enumerate}
        \item in the region $\sqrt{t_2}\ll r_{b_1} $, or $1\ll r_{b_2}$, and so in particular at infinity, 
        $$\Bar{g}_{v_b}=\bar{g}_{v_{b_1}} + t_2^2h^4 + o(t_2^2r_{b_1}^{-4})$$
        \item in the region $r_{b_1}\ll 1$, or $r_{b_2}\ll t_2^{-\frac{1}{2}}$
        $$\Bar{g}_{v_b}=t_2(\bar{g}_{v_{b_1}} + t_2h_2 + t_2^{\frac{3}{2}}h_3 + t_2^2h_4) + o(t_2^2r_{b_2}^4)$$
        \item in the gluing region $\sqrt{t_2}\ll r_{b_1}\ll 1$ or $1\ll r_{b_2}\ll t_2^{-\frac{1}{2}}$, we have
    $$\Bar{g}_{v_b} = \eucl + \mathfrak{H}_2 + \mathfrak{H}_3+\mathfrak{H}^4+\mathfrak{H}_2^4+\mathfrak{H}_3^4+ \mathfrak{H}_4^4+\mathfrak{H}_4 + \mathsf{O}(r_{b_1}^5+ r_{b_2}^{-5}).$$
    \end{enumerate}
\end{cor}

We finish this section by an estimate of the variations at infinity of a tree of singularities with respect to the size of the bubbles attached to it.
\begin{itemize}
        \item for $l\in\{0,...,k+1\}$,  Ricci-flat K\"ahler $g_{b_{l}}$ that can be glued as a tree of singularities according to \cite[Section 4.3]{ozu2}
        \item $g^{B_{k+1}}$ (resp. $g^{B_{k}}$), the na\"{\i}ve gluing of the $g_{b_{l}}$ at scales $T_l = t_0...t_l>0$ for small enough $t_l>0$ as a tree for $l\in\{0,...,k+1\}$ (resp. $l\in\{0,...,k\}$),
        \item $\bar{g}_{B_{k+1}}$ (resp. $\bar{g}_{B_k}$) the  Ricci-flat K\"ahler ALE perturbation of $g^{B_{k+1}}$ (resp. $g^{B_{k}}$),
\end{itemize}

\begin{prop}
    For any $r_{B_k}\gg \sqrt{T_{k+1}}$, we have, for any $l\in \mathbf{N}$, and $C_l>0$ depending only on $l$,
    \begin{equation}\label{control term r-4}
        \rad^l_{B_k}|\nabla^l(\bar{g}_{B_k} - \bar{g}_{B_{k+1}})| \leqslant C_l T_{k+1}^2 r_{B_k}^{-4}.
    \end{equation}

    More precisely, there is a converging expansion
    $$\bar{g}_{B_k} - \bar{g}_{B_{k+1}} = \sum_{j\geqslant 4} \mathfrak{K}^j,$$
    with $|\mathfrak{K}^j|\leqslant C_j' T_{k+1}^2\rad^{-j}$ for $C_j'>0$ independent of $T_{k+1}$ or $r_{B_k}$.
\end{prop}
\begin{proof}[Sketch of proof]
    The result can be easily obtained for the  Ricci-flat K\"ahler ALE metrics asymptotic to $\mathbb{R}^4/\Gamma$ for $\Gamma\subset SU(2)$ cyclic. Indeed, they are Gibbons-Hawking's gravitational instantons  which are \textit{explicitly} given in coordinates \cite{gh}.

    For the general case, one can combine the resolution of singularities described in \cite[Proposition 3.10]{kronquot} together with the description of the asymptotic expansion of the \cite[Proposition 3.14]{kronquot} (see also \cite[Theorem 2.1]{auv}).
\end{proof}

\subsection*{Approximate metrics with trees of singularities}\label{approximate metric tree}

We now explain how to deal with trees of singularities. 

\subsubsection*{Setting up the iterative construction}

In this section, we consider $(X,g_\infty)$, a conformally K\"ahler Einstein orbifold with scalar curvature $12$ with multiple singularities of general type. We then consider a na\"{\i}ve desingularization of  $(X,g_\infty )$ at all of its singularities by trees of singularities which we will denote $g^D_{t,v}$. 

We will order our desingularizations according to their \textit{depth}, that is the size of the bubbles glued.
\begin{enumerate}
    \item we start with $g^{D_0}_{t,v} := g^{A_0}_{t,v}:=\hat{g}^{0}_{t,v}:=\Bar{g}_{v_o}$ as before, and
    \item given $g^{D_k}_{t,v}$, and the next largest scale $T_{k+1}$ at which to glue a Ricci-flat ALE metric $\Bar{g}_{{v_{k+1}}}$ to a remaining singular point $p_{k+1}$,
    $$g^{D_{k+1}}_{t,v} := g^{D_k}_{t,v}\;\#_{p_{k+1},T_{k}^2t_{k+1}}\;\Bar{g}_{v_{k+1}}.$$
\end{enumerate}

The rest of the argument mainly consists in carefully constructing approximations $g^{A_k}_{t,v}$ in order to control the Einstein-modulo-obstructions metrics $\hat{g}^{k}_{t,v}$ and their curvature. We will eventually reach the total Einstein-modulo-obstructions desingularization and be able to control its curvature.

\begin{prop}\label{prop recurrence arbre}
For each  $k\in\{1,...,n\}$,  we have: 
    \begin{enumerate}
        \item at every point, the eignenvalues of the self-dual  curvature $W^+$ of $\hat{g}^{k}_{t,v}$ differ form the eigenvalues of $W^+ (g_\infty)$, 
        at a suitable  point of $X$,  by only  an
         error that uniformly tends to zero  as $(t,v)\to (0,0)$;
        \item $\|\hat{g}^{k}_{t,v} - g^{A_{k}}\|_{C^{2,\alpha}_{\beta,*}}=o(T_k)$, and there is $C>0$ such that for any $l\in\{0,1,2,3\}$; and 
        \begin{equation}\label{controle arbre en r2}
            |\nabla^l(\hat{g}^{k}_{t,v} - g^{D_{k}})|\leqslant C r_{D_{k}}^2.
        \end{equation}
    \end{enumerate}
    
\end{prop}

\subsubsection*{The iterative step}

If we have previously defined $ g^{D_k}_{t,v}$, $g^{A_k}_{t,v}$ and $\hat{g}^{k}_{t,v}$, and $\Tilde{\mathfrak{O}}(g^{A_k}_{t,v})$ in a suitable manner, we  may then define $g^{A_{k+1}}_{t,v}$ so as  to control $\hat{g}^{k+1}_{t,v}$.

If the next considered singular point $p_{k+1}$ is a singular point of $X$, then we define $g^{A_{k+1}}_{t,v}$ and obtain a control of $\hat{g}^{k+1}_{t,v}$ and its curvature exactly as in Section \ref{indicative}. Here, we will therefore consider the remaining case where the singular point $p_{k+1}$ belongs to another Ricci-flat ALE metric of the tree of singularities.

\begin{note}
    In order to simplify  our notation in the remainder of the argument, we will drop (when not necessary) the $(t,v)$-dependence of our tensors and in particular denote $g_{b_i}$ for $\Bar{g}_{v_{b_i}}$.
\end{note}

Assume that:
\begin{itemize}
        \item $g^{D_k}$ and $\hat{g}^k$ defined from the previous steps, satisfy:
        \begin{equation}\label{Phi hatgk tree}
            \mathscr{F}_{g^{D_k}}(\hat{g}^k) = \mathfrak{w}_{0}+\mathfrak{w}_{2}+ \mathsf{O}(T_kr_{b_k}^3) \in \Tilde{\mathfrak{O}}(g^{A_k}),
        \end{equation}
        with $|\mathfrak{w}_{l}|\leqslant 
C_lT_kr_{b_k}^l$ for $C_l>0$,
        \item $g^{D_{k+1}}= g^{D_k} \;\#_{p_{k+1},T_{k}^2t_{k+1}} 
 \; g_{{b_{k+1}}}$ where $g_{{b_{k+1}}}$ is a Ricci-flat ALE metric with a remaining singular point $p_1$, and
        \item $\hat{g}^{k+1}$ the Einstein-modulo-obstructions perturbation of $g^{D_{k+1}} =g^{D_{k+1}}_{t,v} $ transversely to $\Tilde{\mathfrak{O}}(g^{A_k})$.
    \end{itemize}
Our goal will now be to obtain curvature estimates for $\hat{g}^{k+1}$ using  the construction of an approximate Einstein-modulo-obstructions metric. We will in particular show that the above properties stay true by iteration.

For this we will also use the following metrics. 

\begin{itemize}
        \item $g^{B_{k+1}}$, the na\"{\i}ve gluing of the tree of singularity at which $g_{b_{k+1}}$ is glued in $g^{D_{k+1}}$,
        \item $g^{B_{k}}$, the na\"{\i}ve gluing of the tree of singularity at which $g_{b_{k+1}}$ is glued in $g^{D_{k+1}}$, but \textit{without} $g_{b_{k+1}}$, and
        \item $\bar{g}_{B_{k+1}}$ (resp. $\bar{g}_{B_k}$) the  Ricci-flat K\"ahler ALE perturbation of $g^{B_{k+1}}$ (resp. $g^{B_{k}}$).
\end{itemize}

The rest of the section consists in constructing a suitable approximate Einstein-modulo-obstructions metric $g^{A_{k+1}}$ controlling $\hat{g}^{k+1}$ in $o(T_{k+1})$ in order to compensate the $\rad_D^{-2}$ term of our function space. For this, we will leverage the fact that both $\hat{g}^k$ and $\bar{g}_{B_k}$ satisfy \textit{exact} equations in order to rule out errors in powers of $T_k$ alone or in powers of $t_{k+1}$ alone. The following key technical lemma follows from the expansion of Ricci curvature at a metric $g$ in the (small) direction $h$: 
$$\Ric(g+h) = \Ric(g)+(g+h)^{-1}*\Rm(g)+(g+h)^{-2}*\nabla^2h+(g+h)^{-3}*\nabla h *\nabla h,$$ where $*$ corresponds to contractions of tensors.
\begin{lem}\label{lem sum einstein def}
    Let $ g $ be an Einstein metric on an open set $U\subset \mathbb{R}^4$ such that $|\nabla^l_{\eucl}(g-\eucl)|_{\eucl}<\frac{1}{100}$ for $l\leqslant L$. Assume that there exist $2$-tensors $h_1$ and $h_2$ small enough in $C^L(U,\eucl)$ such that $\mathscr{F}_g(g+h_i) = w_i$ for $i\in\{1,2\}$.

    Then, we have the following control: there exists a $C_L>0$, independent of $g,h_1$, and $h_2$, such that, for $l\in \{0,...,\}$,
    $$
 \left|\nabla^l_g\left(\mathscr{F}_g(g+h_1+h_2) - w_1-w_2\right)\right|_g \leqslant
C_L \sum_{l_1+l_2=l \,\text{or}\, (l+2)} |\nabla_g^{l_1}h_1||\nabla_g^{l_2}h_2| ~ . 
$$
\end{lem}
\begin{remark}
    The assumption on $g$  always holds  in suitably rescaled harmonic coordinates, whether our space is a manifold or an orbifold; cf. e.g.  \cite{jos}.
\end{remark}

Near the singular point $p_{k+1}$, we have the following expansions of the relevant  metrics:
\ $$\Bar{g}_{B_k} = \eucl +  \mathfrak{H}_2+ \mathfrak{H}_3+ \mathfrak{H}_4+\mathsf{O}(r_{b_1}^5)$$
    $$T_{k}^{-1}\hat{g}^k = \Bar{g}_{B_k} + \mathfrak{K}_0 + \mathfrak{K}_2 + \mathfrak{K}_3 + \mathfrak{K}_4 + \mathsf{O}(T_kr_{B_k}^5),$$
    where $|\mathfrak{K}_i|\leqslant T_kr_{B_k}^i$.
In particular, by the equation satisfied by $\hat{g}^k_{t,v}$ and \eqref{Phi hatgk tree}, the above tensors satisfy the following equalities:
\begin{itemize}
    \item $P_\eucl \mathfrak{K}_0 =0$
    \item $P_\eucl \mathfrak{K}_2 + Q_\eucl(\mathfrak{K}_0,\mathfrak{H}_2) = \Lambda (\eucl+ \mathfrak{K}_0) + \mathfrak{w}_0$,
    \item $P_\eucl \mathfrak{K}_3 + Q_\eucl(\mathfrak{K}_0,\mathfrak{H}_3) =0$,
    \item $P_\eucl \mathfrak{K}_4 + Q_\eucl(\mathfrak{K}_0,\mathfrak{H}_4) + Q_\eucl(\mathfrak{H}_2, \mathfrak{K}_2) + Q_\eucl(\mathfrak{K}_0,\mathfrak{H}_2,\mathfrak{H}_2) +Q_\eucl(\mathfrak{K}_0,\mathfrak{K}_2,\mathfrak{H}_2) +Q_\eucl(\mathfrak{K}_0,\mathfrak{K}_2,\mathfrak{K}_2) = \Lambda (\mathfrak{H}_2+ \mathfrak{K}_2) + \mathfrak{w}_2$.
\end{itemize}

Extending the results of Section \ref{section first def approx}, we define the following tensors extending the above ones: given the $h_i$ of \eqref{def-h_j}, and $w_0$ and $w_2$ extensions of $\mathfrak{w}_0$ and $\mathfrak{w}_2$ of \eqref{Phi hatgk tree} as in Section \ref{section extension obstructions}, 
\begin{itemize}
    \item $P_{\bar{g}_{v_{b_2}}} k_0 =0$
    \item $P_{\bar{g}_{v_{b_2}}} k_2 + Q_{\bar{g}_{v_{b_2}}}(k_0,h_2) - \Lambda ({\bar{g}_{v_{b_2}}}+ k_0) - w_0 = \mho_2 \in \mathfrak{O}(\bar{g}_{v_{b_2}})$,
    \item $P_{\bar{g}_{v_{b_2}}} k_3 + Q_{\bar{g}_{v_{b_2}}}(k_0,h_3) =0$,
    \item $P_{\bar{g}_{v_{b_2}}} k_4 + Q_{\bar{g}_{v_{b_2}}}(k_0,h_4) + Q_{\bar{g}_{v_{b_2}}}(h_2, h_2) + Q_{\bar{g}_{v_{b_2}}}(k_0,h_2,h_2) +Q_{\bar{g}_{v_{b_2}}}(k_0,k_2,h_2) +Q_{\bar{g}_{v_{b_2}}}(k_0,k_2,k_2) - \Lambda (h_2+ k_2) - w_2 = \mho_4 \in \mathfrak{O}(\bar{g}_{v_{b_2}})$.
    \\
\end{itemize}

Define $\Upsilon_t$ to be a cut-off function supported on $r_{D_k}<2t_{k_0}^{\frac{1}{4}}$ and equal to $1$ on $r_{D_k}<t_{k_0}^{\frac{1}{4}}$ satisfying for some $C_l$ independent of $t$, or $r_{D_k}$, 
\begin{equation}\label{control cutoff t} 
    \nabla^l_{g^D_k} \Upsilon_t \leqslant C_lr_{D_k}^l.
\end{equation} 
Next,  define an approximate Einstein-modulo-obstructions metric as:
$$ g^{A_{k+1}}:= \left(\hat{g}^k + \Upsilon_t(\Bar{g}_{B_{k+1}}-\Bar{g}_{B_k})\right) \widetilde{\#}_{p_{k+1},T_{k},t_{k+1}} \left(T_{k+1}^{-1}\Bar{g}_{B_{k+1}} + k_0+t_{k+1}k_2+t_{k+1}^{3/2}k_3+t_{k+1}^2k_4\right),$$
where $\widetilde{\#}$ is essentially  the same operation introduced in  Definition \ref{def naive desing}, but with the gluing scale $t^{\frac{1}{4}}$ replaced  by $T_{k}^{\frac{1}{2}}t_{k+1}^{\frac{2}{9}}$.
Finally,  define the obstruction space $\Tilde{\mathfrak{O}}(g^{A_{k+1}})$ as the set of $\Tilde{\mathfrak{o}} + \Tilde{\mathbf{w}}$ for $ \mathfrak{o}\in \mathfrak{O}(g_{b_{k+1}})$ and $\mathbf{w}\in \tilde{\mathfrak{O}}(g^{A_k})$, where $\Tilde{\mathfrak{O}}$ and $\Tilde{\mathbf{w}}$ are defined as in Definition \ref{def extension obst} We also extend the definition of $\tilde{\mathfrak{O}}(g^{A_{k+1}}_{t,v})$ exactly as in Section \ref{section extension obstructions} replacing $\Bar{g}_{v_o}$ by $g^{A_{k}}_{t,v}$ and $\mathfrak{O}(g_\infty)$ by $\tilde{\mathfrak{O}}(g^{A_{k}}_{t,v})$.


\subsection*{Approximate metrics with trees of singularities}

We now control how close we are to the actual Einstein-modulo-obstructions metric. This proves the last point of Proposition \ref{prop recurrence arbre}.
\begin{prop}
    We have the following control:
    $$ \|\mathscr{F}_{g^{D_{k+1}}}(g^{A_{k+1}})-\mho^{A_{k+1}}\|_{\rad_D^{-2}C^\alpha_\beta(g^{D_{k+1}})} \leqslant o(T_{k+1}). $$
    Thus, by \cite[Proposition 5.11]{ozu2}, we have 
    $\|\hat{g}^{k+1}_{t,v} - g^{A_{k+1}}\|_{C^{2,\alpha}_{\beta,*}(g^{D_{k+1}})}=o(T_{k+1})$. In particular, since $r_{D_{k+1}}\geqslant T_{k+1}$ there is $C>0$ such that for any $l\in\{0,1,2,3\}$,
    $$
            |\nabla^l(\hat{g}^{k+1}_{t,v} - g^{D_{k+1}})|_{g^{D_{k+1}}}\leqslant C r_{D_{k+1}}^2.$$
\end{prop}
\begin{proof}
There are three regions to consider, and we will control our operator on each of them.

\begin{enumerate}
    \item In the region where we have $g^{A_{k+1}}=\hat{g}^k + \Upsilon_t(\Bar{g}_{B_{k+1}}-\Bar{g}_{B_k})$, thanks to the local controls of Lemma \ref{lem sum einstein def} applied to $g$ \eqref{control term r-4} 
 and \eqref{controle arbre en r2}  we find:
    $$ r_{D_k}^{2+l}\left|\nabla^l_{g^D_k}\left(\mathscr{F}_{g^{D_k}}(g^{A_{k+1}})-\mho^{A_{k+1}}\right)\right|_{g^{D_k}} = \mathsf{O}(T_{k+1}^2 r_{D_k}^{-4} r_{D_k}^2).$$

\item In the region $t_{k_0}^{\frac{1}{4}}<r_{D_k}<2t_{k_0}^{\frac{1}{4}}$, thanks to \eqref{control term r-4} and the controls \eqref{control cutoff t} of the cut-off function, we find:
$$ r_{D_k}^{2+l}\left|\nabla^l_{g^D_k}\left(\mathscr{F}_{g^{D_k}}(g^{A_{k+1}})-\mho^{A_{k+1}}\right)\right|_{g^{D_k}} = \mathsf{O}(T_{k+1}^2 r_{D_k}^{-4}) = \mathsf{O}(T_{k+1}^2 t_{k_0}) = o(T_{k+1}).$$

    \item Since   $T_{k+1}^{-1}g^A_{t,v}=T_{k+1}^{-1}\Bar{g}_{B_{k+1}} + k_0+t_{k+1}k_2+t_{k+1}^{3/2}k_3+t_{k+1}^2k_4$, using \eqref{controle arbre en r2}, we find:
    $$ r_{D_{k+1}}^{2+l}\left|\nabla^l_{g^D_{k+1}}\left(\mathscr{F}_{g^{D_{k+1}}}(g^{A_{k+1}})-\mho^{A_{k+1}}\right)\right|_{g^{D_{k+1}}} = \mathsf{O}(T_{k}  t_{k+1}^{\frac{5}{2}}r_{b_{k+1}}^5).$$
    Here, the error comes from multilinear terms involving the following tensors such that for all $l\leqslant 3$, there exists $C_l>0$ so that $$r_{b_{k+1}}^l|\nabla^l_{b_{k+1}}(t_{k+1}^{\frac{i}{2}} k_i)|_{g_{b_{k+1}}}\leqslant  T_{k}t_{k+1}^{\frac{i}{2}}r_{b_{k+1}}^i$$ by \eqref{controle arbre en r2}, and $r_{b_{k+1}}^l|\nabla^l_{g_{b_{k+1}}}(t_{k+1}^{-1}\Bar{g}_{B_{k+1}}- g_{b_{k+1}})|_{g_{b_{k+1}}}\leqslant t_{k+1}r_{b_{k+1}}^2$ by Proposition \ref{prop controle KRFALE}.

    \item Lastly, in the cut-off region $T_{k}^{\frac{1}{2}}t_{k+1}^{\frac{2}{9}}<r_{D_{k+1}}<2T_{k}^{\frac{1}{2}}t_{k+1}^{\frac{2}{9}}$, the difference of the metrics is in:
    $$ T_kr_{b_k}^5 + T_kr_{b_{k+1}}^{-4}.$$
\end{enumerate}

Putting all of this together, we consequently obtain  the stated estimate by gluing at scale $t_{k+1}^{\frac{1}{5}}\ll r_{b_{k+1}}\ll t_{k+1}^{\frac{1}{4}}$, which is satisfied by the scale of gluing chosen in the definition of $g^{A_{k+1}}$.
\end{proof}

Mimicking the proof of the third point of Proposition \ref{controle desing partielle} now yields the following result.

\begin{cor}\label{controle curvature step k to k+1}
   If the eigenvalues  of the self-dual  Weyl curvature $W^+$ of $\hat{g}^k$ are, are   uniformly  close to those of $g_\infty$ at some corresponding  point of $X$, 
       then  the  eigenvalues  of the self-dual  Weyl curvature $W^+$  of $\hat{g}^{k+1}$  also have this property,  with  an error that uniformly  tends to zero as $(t,v)\to 0$.
\end{cor}

\subsection*{Conclusion of the proof}

Let us now assume  that the $\lambda >0$ Einstein orbifold $(X,g_\infty )$ is 
 conformally K\"ahler,   and that there is 
an admissible  sequence $(M,g_i)$ with 
    $$d_{GH}\Big((M,g_i),(X,g_\infty)\Big)\to 0.$$
    By \cite{ozu1}, after passing to a subsequence indexed by $j$, and for sufficiently large $j$,  there thus exist na\"{\i}ve desingularizations
 $g^D_{t_j}$ such that $\|g_j-g^{D}_{t_j}\|_{C^{2,\alpha}_{\beta,*}(g^D_{t_j})} \to 0$, possibly with the involvement of 
 trees of singularities.
After perhaps refining our subsequence,  the associated gluing  scales  $T_{k,j}>0$ in the tree  can then be ordered such that, for all $j$,
$$T_{1,j}\leqslant T_{2,j}\leqslant ... \leqslant T_{n,j}.$$
By \cite{ozu1,ozu2}, each $g_j$ is then isometric to some such  $\hat{g}^n_{t_j,v_j}$ for $j \gg 0$. 

However, the two lowest eigenvalues of $W^+$ are everywhere negative and equal 
for the conformally K\"ahler $\lambda >0$ Einstein orbifold metric $g_\infty$. By Corollary \ref{controle curvature step k to k+1},
it therefore follows, for $j\gg 0$,  that the two lowest eigenvalues of the self-dual Weyl curvatures $W^+$ of our Einstein metrics $g_j$ 
are therefore  everywhere negative.  Since $W^+\in \End (\Lambda^+)$ is  trace-free, this means that these  $g_j$ must have   $\det (W^+) > 0$
everywhere. Thus, all but a finite number of the Einstein metrics $g_j$  must actually  satisfy  Wu's criterion \eqref{pengwu}.

It is therefore   a formal consequence that 
  all but a finite number of the original $g_i$ must  also satisfy \eqref{pengwu}.  Indeed, there would  otherwise be a subsequence $\{ g_{j^\prime}\} \subset 
\{ g_i\}$ consisting of Einstein metrics that {\em  do not}  satisfy \eqref{pengwu}. But since the  $\{ g_{j^\prime}\}$ would still  be an admissible sequence converging
to $(X, g_\infty )$,   the above argument  would then demonstrate   that \eqref{pengwu} must  hold for an infinite number of the $\{ g_{j^\prime}\}$, and this is of course
a contradiction. 
This shows   that,  for all  $i \gg 0$,  the original $\{ g_i\}$ must actually  satisfy \eqref{pengwu}, and must therefore \cite{lebdet,pengwu}  be conformally K\"ahler. 

 The classification \cite[Theorem A]{lebuniq} of conformally K\"ahler,  Einstein metrics on smooth, compact $4$-manifolds therefore tells us that these  $(M,g_i)$ must  actually 
 be K\"ahler-Einstein unless $M$ is either $\CP_2\# \overline{\CP}_2$ or $\CP_2\# 2\overline{\CP}_2$, equipped with a constant multiple of the 
  Page metric or the Chen-LeBrun-Weber metric, respectively. 
  Since these exceptional  metrics  are also, up to isometry and rescaling,   the {\em only} conformally K\"ahler, Einstein metrics on
  these exceptional manifolds,  the  $(M,g_i)$ must in fact be K\"ahler-Einstein for $i \gg 0$ if we assume that $(X,g_\infty )$
  either has an orbifold singularity, or is itself K\"ahler-Einstein. Theorems \ref{alpha} and \ref{beta} thus now follow, as promised. 
  \hfill $\Box$

\pagebreak 
\appendix

\section{Function Spaces}\label{function spaces}
    For a tensor $s$, a point $x$, $\alpha>0$ and a Riemannian manifold $(M,g)$. The Hölder seminorm is defined as
$$ [f]_{C^\alpha(g)}(x):= \sup_{\{y\in T_xM,|y|< \textup{inj}_g(x)\}} \Big| \frac{f(x)-f(\exp^g_x(y))}{|y|^\alpha} \Big|_g.$$

For orbifolds, we consider a norm which is bounded for tensors decaying at the singular points.
\begin{defn}[Weighted Hölder norms on an orbifold]\label{norme orbifold}
    Let $\beta\in \mathbb{R}$, $k\in\mathbb{N}$, $0<\alpha<1$ and  $(X,g_\infty )$ an orbifold. Then, for any tensor $s$ on $X$, we define
    \begin{align*}
        \| s \|_{C^{k,\alpha}_{\beta}(g_\infty)} &:= \sup_{X}\rad_o^{-\beta}\Big(\sum_{i=0}^k \rad_o^{i}|\nabla_{g_\infty}^i s|_{g_\infty} + \rad_o^{k+\alpha}[\nabla_{g_\infty}^ks]_{C^\alpha(g_\infty)}\Big).
    \end{align*}
\end{defn}

For ALE manifolds, we will consider a norm which is bounded for tensors decaying at infinity.

\begin{defn}[Weighted Hölder norms on an ALE orbifold]\label{norme ALE}
Let $\beta\in \mathbb{R}$, $k\in\mathbb{N}$, $0<\alpha<1$ and $(N,g_b)$ be an ALE manifold. Then, for any tensor $s$ on $N$, we define
   \begin{align*}
       \| s \|_{C^{k,\alpha}_{\beta}(g_b)}:= \sup_{N}\rad_b^\beta\Big( \sum_{i=0}^k\rad_b^{i}|\nabla_{g_b}^i s|_{g_b} + \rad_b^{k+\alpha}[\nabla_{g_b}^ks]_{C^\alpha({g_b})}\Big).
   \end{align*}
\end{defn}

Next, using a partition of unity 
$$1= \Upsilon_{X^t} + \sum_j \Upsilon_{N_j^t}$$
on $M$ , 
where $\Upsilon_{X^t}$ equals  $1$ when $g^D_t=g_\infty$ and $\Upsilon_{N_j^t}$ equals  $1$ where $g^D_t = t_j g_{b_j}$ (see \cite{ozu1}), we can   define a global norm.
\begin{defn}[Weighted Hölder norm on a na\"{\i}ve desingularization]\label{norme a poids M}
		Let $\beta\in \mathbb{R}$, $k\in\mathbb{N}$ and $0<\alpha<1$. We define for $s\in TM^{\otimes l_+}\otimes T^*M^{\otimes l_-}$ a tensor $(l_+,l_-)\in \mathbb{N}^2$, with $l:= l_+-l_-$ the associated conformal weight,
		$$ \|s\|_{C^{k,\alpha}_{\beta}(g^D)}:= \| \Upsilon_{X^t} s \|_{C^{k,\alpha}_{\beta}(g_\infty)} + \sum_j T_j^{\frac{l}{2}}\|\Upsilon_{N_{j}^t}s\|_{C^{k,\alpha}_{\beta}(g_{b_j})}.$$
\end{defn}

\subsubsection*{Decoupling norms.}

 For each singular point $p_k$, we consider the annuli 
 $A_k(t,\epsilon_0) := (\Phi_k)_*A_{\eucl}(\epsilon_0^{-1}\sqrt{t_k},\epsilon_0)$ and $B_k(\epsilon_0):=(\Phi_k)_*B_{\eucl}(0,\epsilon_0)$, as well as cut-off functions $\Upsilon_{A_k(t,\epsilon_0)}$ and $\Upsilon_{B_k(\epsilon_0)}$ respectively supported in $A_k(t,\epsilon_0)$ and $B_k(\epsilon_0)$, and equal to $1$ on $A_k(t,2\epsilon_0)$ and $B_k(\epsilon_0/2)$.

\begin{defn}[$C^{k,\alpha}_{\beta,*}$-norm on $2$-tensors] 
    Let $ h $ be a $2$-tensor on $(M,g^D)$,  $(X,g_\infty )$ or $(N,g_b)$. We define its $C^{k,\alpha}_{\beta,*}$-norm by
    $$\|h\|_{C^{k,\alpha}_{\beta,*}}:= \inf_{h_*,\mathfrak{H}_k} \|h_*\|_{C^{k,\alpha}_{\beta}} + \sum_k |\mathfrak{H}_k|_{\eucl},$$
    where the infimum is taken on the $(h_*,(\mathfrak{H}_k)_k)$ satisfying $h= h_*+\sum_k \Upsilon_{A_k(t,\epsilon_0)}\mathfrak{H}_k$ for $(M,g^D)$ or $h= h_*+\sum_k \Upsilon_{B_k(\epsilon_0)}\mathfrak{H}_k$ for  $(X,g_\infty )$ or $(N,g_b)$, where each $\mathfrak{H}_k$ is some constant and trace-free symmetric $2$-tensors on $\mathbb{R}^4\slash\Gamma_k$.
\end{defn}

 Finally, we also introduce a function space that is particularly well-suited  to Ricci-flat ALE spaces; cf.  \cite[Definition 5.18]{ozuthese}.

\begin{defn}[$C^{2,\alpha}_{\beta,**}$ norm on an ALE orbifold]\label{norm orbifold ALE}
    Let $(N,g_b)$ be an ALE orbifold, let $h$ be a $2$-tensor on $N$, and assume that $h = \mathfrak{H}^4 + \mathsf{O}(\rad_b^{-4-\beta})$ for $\beta>0$. For $\Upsilon$, a cut-off function supported in $N -  K$.
     We define the $C^{2,\alpha}_{\beta,**}$-norm of $h$ as:
    $$\|h\|_{C^{2,\alpha}_{\beta,**}}:= \sup \rad_b^{4}|\mathfrak{H}^4|_{g_\eucl} + \big\|(1+\rad_b)^{4}(h-\Upsilon \mathfrak{H}^4)\big\|_{C^{2,\alpha}_{\beta,*}}.$$
\end{defn}
Note that this definition  actually  corrects a minor  error in \cite{ozu2} that is fortunately   harmless in most contexts. For a detailed
discussion of this point, see   \cite[Remark A.12]{huoz} concerning   Ricci-flat  K\"ahler ALE metrics.

\vfill

\noindent {\bf Acknowledgments.} The authors would like to  thank the Simons Center for Geometry and Physics for its hospitality during the completion  of 
 this article. The first author would also like to express his gratitude to the Simons Foundation and the Fondation Sciences Math\'ematiques de Paris for their partial financial support
of this research.

\pagebreak

%
%
%

\end{document}